\documentclass[12pt,oneside]{ip-journal}
\usepackage{geometry}
\makeatletter
\usepackage[all,cmtip]{xy}
\usepackage{mathpazo}
\usepackage{epstopdf}
\usepackage{mathtools}
\usepackage[cmtip,all]{xy}
\usepackage{cancel}

\usepackage{mathrsfs}
\usepackage{amscd}
\usepackage{amsmath}
\usepackage{amssymb}
\usepackage{amsthm}
\usepackage{float}
\usepackage[dvips]{graphicx}
\usepackage{color}
\usepackage{tikz} 
\usepackage{enumerate}

\usetikzlibrary{matrix,arrows}

\usepackage{array}
\usepackage{amscd}
\usepackage[all]{xy}

\DeclareFontFamily{U}{rsfs}{%
\skewchar\font127}
\DeclareFontShape{U}{rsfs}{m}{n}{%
<-6>rsfs5<6-8.5>rsfs7<8.5->rsfs10}{}
\DeclareSymbolFont{rsfs}{U}{rsfs}{m}{n}
\DeclareSymbolFontAlphabet
{\mathrsfs}{rsfs}
\DeclareRobustCommand*\rsfs{%
\@fontswitch\relax\mathrsfs}

\DeclareFontFamily{U}{rsfs}{%
\skewchar\font127}
\DeclareFontShape{U}{rsfs}{m}{n}{%
<-6>rsfs5<6-8.5>rsfs7<8.5->rsfs10}{}
\DeclareSymbolFont{rsfs}{U}{rsfs}{m}{n}
\DeclareSymbolFontAlphabet
{\mathrsfs}{rsfs}
\DeclareRobustCommand*\rsfs{%
\@fontswitch\relax\mathrsfs}

\theoremstyle{plain}

\newtheorem{thm}{Theorem}[section]
\newtheorem{prop}[thm]{Proposition}
\newtheorem{lem}[thm]{Lemma}

\newtheorem{defi}[thm]{Definition}
\newtheorem{rmk}[thm]{Remark}
\newtheorem{cor}[thm]{Corollary}

\newtheorem*{prop*}{Proposition}
\newtheorem*{notn}{Notation}
\newtheorem*{thr}{Theorem}
\newtheorem{prop-defi}[thm]{Proposition-Definition}
\newtheorem{thm-defi}[thm]{Theorem-Definition}
\newtheorem{lem-defi}[thm]{Lemma-Definition}

\newtheorem{assum}[thm]{Assumption}

\newcommand{\C}{\mathbb{C}}
\newcommand{\ZZ}{\mathbb{Z}}

\newcommand{\I}{\mathscr{I}}

\renewcommand{\det}{\operatorname{det}}

\renewcommand{\O}{\mathscr{O}}

\renewcommand{\L}{{\boldsymbol{L}}}
\renewcommand{\S}[1]{S^{[#1]}}

\renewcommand{\i}[1]{\I^{[#1]}}

\newcommand{\id}{\operatorname{id}}
\newcommand{\pr}{\operatorname{pr}}
\newcommand{\LL}{\mathbb{L}^{\bullet}}

\newcommand{\Fb}{F^{\bullet}}

\newcommand{\Eb}{E^{\bullet}}
\newcommand{\Gb}{G^{\bullet}}
\newcommand{\Fbv}{F^{\bullet \vee}}
\newcommand{\Fbr}{F^{\bullet}_{\operatorname{rel}}}
\newcommand{\Fbrv}{F^{\bullet \vee}_{\operatorname{rel}}}
\newcommand{\dR}{\mathbf{R}}
\newcommand{\dL}{\mathbf{L}}
\renewcommand{\hom}{\mathscr{H}om}
\newcommand{\ext}{\mathscr{E}xt}
\newcommand{\Hom}{\operatorname{Hom}}
\newcommand{\Ext}{\operatorname{Ext}}
\newcommand{\cone}{\operatorname{Cone}}
\newcommand{\Coh}{\operatorname{Coh}}
\newcommand{\tr}{\operatorname{tr}}

\newcommand{\q}{\operatorname{q}}
\newcommand{\coker}{\operatorname{coker}}
\newcommand{\at}{\operatorname{At}}
\newcommand{\Tb}{\overline{T}}

\newcommand{\rank}{\operatorname{Rank}}
\newcommand{\mM}{\mathscr{M}}
\newcommand{\mMb}{\mathscr{M}^{\mathscr{L}}}
\newcommand{\mMw}{\mathscr{M}^{\omega_S}}
\newcommand{\cE}{\mathscr{E}}
\newcommand{\eE}{\mathbb{E}}
\newcommand{\eEb}{\overline{\mathbb{E}}}
\newcommand{\ch}{\operatorname{ch}}
\newcommand{\DT}{\operatorname{DT}}
\newcommand{\VW}{\operatorname{VW}}
\newcommand{\vir}{\operatorname{vir}}

\newcommand{\F}{\mathscr{F}}

\renewcommand{\t}{\mathbf{t}}
\newcommand{\s}{\mathbf{s}}

\newcommand{\G}{\mathscr{G}}
\newcommand{\fix}{\operatorname{fix}}
\newcommand{\mov}{\operatorname{mov}}
\newcommand{\Ebf}{E^{\bullet, \fix}}
\newcommand{\Ebm}{E^{\bullet,\mov}}

\newcommand{\red}{\operatorname{red}}
\newcommand{\rel}{\operatorname{rel}}
\newcommand{\ob}{\operatorname{ob}}
\newcommand{\cS}{\mathcal{S}_{B}}

\newcommand{\cL}{\mathscr{L}}

\newcommand{\bp}{\overline{p}}

\title[Higher rank flag sheaves on Surfaces]{Higher rank flag sheaves on Surfaces and Vafa-Witten invariants}

\author{Artan Sheshmani and Shing-Tung Yau}

\begin{document}
\maketitle
\begin{abstract}
We study moduli space of holomorphic triples $E_{1}\xrightarrow{\phi} E_{2}$, composed of torsion-free sheaves $E_{i}, i=1,2$ and a holomorphic mophism between them, over a smooth complex projective surface $S$. The triples are equipped with Schmitt stability condition \cite{S00}. We observe that when Schmitt stability parameter $q(m)$  becomes sufficiently large, the moduli space of triples benefits from having a perfect relative and absolute deformation-obstruction theory in some cases. We further generalize our construction by gluing triple moduli spaces, and extend the earlier work \cite{GSY17a} where the obstruction theory of nested Hilbert schemes over the surface was studied. Here we extend the earlier results to the moduli space of  chains$$E_{1}\xrightarrow{\phi_{1}} E_{2}\xrightarrow{\phi_{2}} \cdots \xrightarrow{\phi_{n-1}} E_{n},$$ where $\phi_{i}$ are injective morphisms and $rk(E_{i})\geq 1, \forall i$. There is a connection, by \textit{wallcrossing in the master space}, between the theory of such higher rank flags, and the theory of Higgs pairs on the surface, which provides the means to relate the flag invariants to the local DT invariants of threefold given by a line bundle on the surface, $X :=Tot(\cL\to S)$. The latter, when $\cL=\omega_{S}$, provides the means to compute the contribution of higher rank flag sheaves to partition function of Vafa-Witten theory on $X$.\\
\smallskip



\end{abstract}

\tableofcontents
\section{Introduction} 

This paper is an attempt to generalize some  of the recent developments in Donaldson-Thomas theory, namely the mathematical theory of 
Vafa-Witten invariants of complex surfaces constructed in \cite{TT17a,TT17b}, and its connections to local Donaldson-Thoms theory developed in \cite{GSY17a, GSY17b}. In the approach of \cite{GSY17b, TT17a} the Vafa-Witten invariants of a smooth projective surface $S$ are constructed as the reduced DT virtual cycle invariants of two dimensional sheaves on the total space of the canonical bundle $K_S$. In fact, in more generality the construction of loc. cit. involves equivariant residual invariants of  two dimensional sheaves on a noncompact threefold $X$, given by the total space of an arbitrary line bundle $\cL$ on $S$. The authors in \cite{GSY17b} showed that the local DT invariants of $X$, in some cases, enjoy receiving nontrivial contributions from invariants of nested Hilbert schemes, which parametrize the nested flags of ideal sheaves of (non-pure) one dimensional subschemes of $S$. In fact the study of local DT theory of local-surface threefolds, was the main reason for the authors in \cite{GSY17a} to develop a well-behaved deformation-obstruction theory for nested Hilbert schemes on complex smooth projective surfaces (as an intrinsic surface theory) and compute their virtual cycle induced invariants in some manageable cases. The current article has a similar birth story, and the study of local DT theory of noncompact local-surface threefolds has made it clear that a construction is needed, again as an intrinsic surface theory, for a well behaved deformation-obstruction theory of flags $$E_{1}\xrightarrow{\phi_{1}} E_{2}\xrightarrow{\phi_{2}} \cdots \xrightarrow{\phi_{n-1}} E_{n},$$in which the torsion-free sheaves $E_{i}$ might not be ideal sheaves necessarily,  and rather can have higher rank on their support, $S$. \\

In Section \ref{chap1} we start with the most basic version of flags and discuss the theory of stable triples. These are the data of two torsion-free holomorphic sheaves on $S$, $(E_{1}, E_{2})$, and a holomorphic map, $\phi$, between them. We equip these tuples with Schmitt's stability condition for $A_{n}$ quiver representations (Definition \ref{defi8}). The Schmitt stability criteria is written in terms of Hibert polynomial for the quiver and comes with parameters $q_{1}$ and $q_{2}$, given both as polynomials (one would need as many parameters as the number of vertices in a given $A_{n}$ quiver). We then equate the parameters $q_{1}=q_{2}=q$ and discuss that the stability criteria becomes ``\textit{nice}" with this specialization, particularly when $q$ is given by a polynomial with positive \textit{large enough} leading coefficient. We call stability with respect to this choice of $q$, the marginal stability, and refer to the corresponding triples as $q_{+_{\gg}}$-stable triples. Lemma \ref{rk1} is simple but a key lemma in the paper, as it shows that the $q_{+_{\gg}}$-stable triples $(E_{1},E_{2}, \phi)$ must satisfy the property that the morphism $\phi$ does not have a kernel. The importance of this fact is that, in construction of deformation obstruction theory for holomorphic triples later, the vanishing of kernel guarantees the perfectness of their deformation-obstruction complex (in particular the vanishing of its cohomology in degree $-2$). \\

In Section \ref{chap2} we recall the construction of moduli space of holomorphic triples. In Section \ref{More} we discuss some relations between $q_{+_{\gg}}$-stability and other notions of stability conditions. In particular we perform several changes of variable from $q$ to $q'$ (Definition \ref{q'}) and then to $\sigma$ (Definition \ref{sigma}) and re-write the $q$-stability criteria for triples in terms of $\sigma$-stability criteria. The motivation to change variables is that, it is easier to see that the $\sigma$-stability relates to Higgs stability of Higgs pairs; that is, assuming to start from an $\cL$-twisted ``\textit{split}"  Higgs pair
\begin{equation}\label{Higgs-bundle}
\bigg(E_{1}\oplus E_{2}\bigg)\xrightarrow{\psi:=\bigg(\begin{array}{cc}0 & 0\\ \phi & 0\end{array}\bigg)}\bigg(E_{1}\oplus E_{2}\bigg)\otimes \mathscr{L}
\end{equation}
for a line bundle $\cL$ on $S$, we can naturally associate it to a holomorphic triple $E_{1}\xrightarrow{\phi}E_{2}\otimes \cL$ , and show that the Higgs stability of the pair is equivalent to the $\sigma$-stability of the triple when the value of $\sigma$ is a certain critical value. In fact it can be shown that for such particular Higgs pairs (or one can view them as their equivalent triples), a continuous perturbation of the parameter $\sigma$ from large enough values to this critical value leads to a jumping phenomenon in the behavior of the moduli space, which can be formulated using the machinery of wallcrossing in the master space developed by Mochizuki in \cite{M09} (Remark \ref{l-higgs}). We analyze this wallcrossing in detail in a sequel to this article \cite{SY20}. \\

We point out that not every Higgs pair on the surface is split as in \eqref{Higgs-bundle}. Hence in order to motivate the study of such pairs, we make a digression in Section \ref{sec:threefold} to the theory of compactly supported coherent sheaves on local surface threefold $X:=Tot(\cL\to S)$ and their connection to such split Higgs pairs on $S$. We define a slope stability condition for torsion sheaves with set-theoretic support on $S$, and in Section \ref{compact-moduli} study the stable locus of their parameterizing space. Their moduli space is a priori noncompact and hence the Graber-Pandharipande \cite{GP99} method of virtual localization is needed in order to define the DT invariants of such moduli space via residue integration over the fixed loci induced by the $\C^*$ action that dilates the fibers of $\cL$. We include a construction of a suitable reduced obstruction theory, borrowed from \cite{GSY17b}, which uses the Behrend and Fantechi approach \cite{BF97}, and show that the obtained reduced obstruction theory governs the $\C^*$-fixed obstruction theory for the fixed loci of moduli space of stable torsion sheaves on $X$ (Corollary \ref{red-fixed}). In section \ref{fixed-loci} we study the fixed loci, which leads to the connection between the stable $\C^*$-equivariant sheaves on $X$ and stable ($\C^*$-graded) split $\cL$-twisted Higgs pairs on $S$ given as:

\begin{equation}\label{Higgs-bundle'}
\tiny{\bigg(E_{1}\oplus\cdots \oplus E_{n}\bigg)\xrightarrow{\psi:=\left[\begin{array}{cccccc}0 & 0 & 0& \cdot &0 \\
\phi_{1} & 0& 0 & \cdot & \cdot \\ 
0 & \phi_2  & 0 & \cdot & \cdot  \\
\cdot & \cdot & \dots & \cdot & \cdot   \\
\cdot & \cdot & \dots & \cdot & \cdot  \\
0 & 0& \dots & \phi_{n-1} &  0  \\
 \end{array}\right]}\bigg(E_{1}\oplus\cdots \oplus E_{n}\bigg)\otimes \mathscr{L}}
\end{equation}

 The latter Higgs pairs are the same as $\sigma_{crit}$-stable flags by the analysis in Section \ref{chap2}, which (given the equivalence between $\sigma$-stability and $q$-stability via change of variable from $\sigma$ to $q$) we can re-write as $q_{crit}$-stable flags.\\

In a sequel to current article we apply (and adapt) Mochizuki's machinery of wallcrossing in the master space \cite{M02} to the moduli space of split Higgs pairs induced as the $\C^*$-fixed loci of the moduli space of compactly supported sheaves on $X$, and relate the invariants of $q_{+_{\gg}}$-stable flags to invariants of $q_{crit}$-stable flags. Sections \ref{sec:threefold} and \ref{fixed-loci} serve merely as segments providing the motivation for study of higher rank flags in this article, and could be skipped by a reader who is merely interested in the construction of deformation obstruction theory for flags, purely as a surface theory.

In Section  \ref{def-obs} we start from length 2 flags and construct, in two steps, an absolute deformation-obstruction theory for $q_{+_{\gg}}$-stable triples $\phi: E_{1}\to E_{2}$. The first step is to construct a relative perfect obstruction theory of triples. Relative here means that we construct an obstruction theory for triples in which the torsion-free sheaf $E_{1}$ is  not allowed to deform. The moduli space of such rigidified triples is realized as the fiber of the forgetful map from the moduli space of all triples to the moduli space parametrizing $E_{1}$. The central theorem of Section \ref{def-obs} is Theorem \ref{rel}, where we use the reduced Atiyah classes of Illusie and $q_{+_{\gg}}$-stability to show that the induced relative obstruction theory is perfect of correct amplitude $[-1,0]$. \\

The second step of our construction is explained in Section \ref{2-step}, Theorem \ref{abs}, where we use the cone construction of Maulik-Pandharipande-Thomas \cite{MPT10} to obtain an absolute obstruction theory for triples. Here, we assume that the moduli space of sheaves, $E_{1}$, is given as a smooth projective scheme. The latter is possible when $H^{1}(S, \O_S)\cong 0$, $K_{S}\leq 0$, and $E_{1}$ satisfies the condition that $gcd(rk(E_{1}), deg(E_{1}))=1$, or when $E_{1}$ is given as ideal sheaf of zero-dimensional subschemes of $S$ with given fixed length. The smoothness property of the moduli space parameterizing $E_{1}$ is a restriction imposed on us, due to the lack of sufficient technical machinery to handle more general situations. The first named author is currently working on generalizing the construction in this article, in an example, to the case where the base of fibration is non-smooth, and yet equipped with a perfect deformation-obstruction theory of amplitude $[-1,0]$. The latter construction is beyond the scope of current article and will be discussed later in a sequel.\\
 
In Section \ref{high-length} we generalize our constructions to higher length flags. Here we construct the moduli space of higher length flags, as an intrinsic moduli theory over the surface\footnote{That is, we do not necessarily view them as  induced $\C^*$-fixed loci of ``\textit{threefold moduli theory}" for a noncompact local surface threefold.}, by gluing the moduli spaces of $q_{+_{\gg}}$-triples, however such moduli spaces can also be realized as the induced $\C^*$-fixed loci of some moduli space of stable torsion sheaves on the threefold $X$ considered in Sections \ref{sec:threefold} and \ref{fixed-loci}. In Section \ref{sec:reduced} in cases where it becomes necessary, we construct the reduced obstruction theory for moduli spaces of triples (and similarly higher length flags) using Kiem-Li co-section localization machinery. An example of such cases is when $p_{g}(S)>0$ and the flag $$E_{1}\xrightarrow{\phi_{1}} E_{2}\xrightarrow{\phi_{2}} \cdots \xrightarrow{\phi_{n-1}} E_{n}$$ consists of $E_{i}, 1<i<m\leq n$ where $rk(E_{i})=1$. In the sequel to current article we will compute an example of our invariants of higher length flag sheaves over smooth curves of genus $g$ as well as some toric surfaces such as $\mathbb{P}^2$ and $\mathbb{P}^1 \times \mathbb{P}^1$, then we generalize our construction to all smooth projective surfaces.

\section*{Aknowledgement}
The first named author would like to thank Amin Gholampour and Martijn Kool, for many valuable conversations, specially commenting on first versions of this article. He also is in great debt to Andrei Negut, whom he bothered constantly and learned from his great insight on algebraic geometry and representation theory of flags. The first author would further like to sincerely thank Steven Bradlow for teaching him about joint work with Garcia Prada on moduli space of triples on Riemann surfaces. Further people to thank are: Sergey Arkhipov, Vladimir Baranovsky, Gergely Berczi, Simon Donaldson, Alexander Efimov, Peter Gothen, An Huang, Sheldon Katz, Bong Lian, Melissa Liu, Alina Marian, Dragos Oprea, Christiano Spotti, and finally very special thanks go to Emanuel Diaconescu for teaching the authors his insight about many aspects of current project, and specially, the moduli space of ADHM sheaves. A. S. was partially supported by NSF DMS-1607871, NSF DMS-1306313, Simons 38558, and Laboratory of Mirror Symmetry NRU HSE, RF Government grant, ag. No 14.641.31.0001. Finally, the first author would like to sincerely thank Center for Mathematical Sciences and Applications at Harvard University, as well as the center for Quantum Geometry of Moduli Spaces at Aarhus University, and the Laboratory of Mirror Symmetry in Higher School of Economics, Russian federation, for their great help and support. S.-T. Y. was partially supported by NSF DMS-0804454, NSF PHY-1306313, and Simons 38558.

\section{Stable triples}\label{chap1}
\begin{defi}\label{defi3}(\textit{Holomorphic triples})
Let $S$ be a nonsingular projective surface over $\mathbb{C}$. A holomorphic triple supported on $S$ is given by $(E_{1},E_{2},\phi)$ consisting of  a pair of coherent sheaves $E_{1}$ and  $E_{2}$, together with a holomorphic morphism $\phi:E_{1} \rightarrow E_{2}$.
A homomorphism of triples from $(E'_{1},E'_{2},\phi')$ to $(E_{1},E_{2},\phi)$ is a commutative diagram:
\begin{center}
\begin{tikzpicture}
back line/.style={densely dotted}, 
cross line/.style={preaction={draw=white, -, 
line width=6pt}}] 
\matrix (m) [matrix of math nodes, 
row sep=1em, column sep=3.5em, 
text height=1.5ex, 
text depth=0.25ex]{  
E_{1}'&E'_{2}\\
E_{1}&E_{2}\\};
\path[->]
(m-1-1) edge node [above] {$\phi'$} (m-1-2)
(m-1-1) edge (m-2-1)
(m-1-2) edge (m-2-2)
(m-2-1) edge node [above] {$\phi$} (m-2-2);
\end{tikzpicture}
\end{center}
Now let $T$ be a $\mathbb{C}$-scheme of finite type and let $\pi_{S}:S \times T \rightarrow S$ and $\pi_{T}: S\times T \rightarrow T$ be the corresponding projections. A \textit{$T$-flat family} of triples on $S$ is a triple $(\mathscr{E}_{1},\mathscr{E}_{2},\phi)$ consisting of a morphism of $\mathscr{O}_{S\times T}$ modules $\mathscr{E}_{1}\xrightarrow{\phi} \mathscr{E}_{2}$ such that $\mathscr{E}_{1}$ and $\mathscr{E}_{2}$ are flat over $T$ and for every point $t\in T$ the fiber  $(\mathscr{E}_{1},\mathscr{E}_{2},\phi)\mid_{t}$ is given by a holomorphic triple $(E_{1},E_{2}, \phi)$ on $S$. Two $T$-flat families of triples $(\mathscr{E}_{1},\mathscr{E}_{2},\phi)$ and $(\mathscr{E}'_{1},\mathscr{E}'_{2},\phi')$ are isomorphic if there exists a commutative diagram  of the form:
\begin{center}
\begin{tikzpicture}
back line/.style={densely dotted}, 
cross line/.style={preaction={draw=white, -, 
line width=6pt}}] 
\matrix (m) [matrix of math nodes, 
row sep=1em, column sep=3.5em, 
text height=1.5ex, 
text depth=0.25ex]{ 
\acute{\mathscr{E}}_{1}&\acute{\mathscr{E}}_{2}\\
\mathscr{E}_{1}&\mathscr{E}_{2}\\};
\path[->]
(m-1-1) edge node [above] {$\phi'$} (m-1-2)
(m-1-1) edge node [left] {$\cong$} (m-2-1)
(m-1-2) edge node [right] {$\cong$} (m-2-2)
(m-2-1) edge node [above] {$\phi$} (m-2-2);
\end{tikzpicture}
\end{center}
\end{defi}
\begin{defi}(\textit{Type of a triple})
A triple of type $(\overrightarrow{v}_{1},\overrightarrow{v}_{2})$ is given by a triple $(E_{1}, E_{2},\phi)$ such that   $\text{Ch}(E_{1})=\overrightarrow{v}_{1}$ and $\text{Ch}(E_{2})=\overrightarrow{v}_{2}$. Note that by the Grothendieck-Riemann-Roch theorem, fixing the Chern character vector of sheaves $E_{1},E_{2}$ is equivalent to fixing their Hilbert polynomials: $P_{E_{1}}, P_{E_{2}}$ respectively. 
\end{defi}

\begin{defi}(\textit{Stability of holomorphic triples})\label{defi8}
 Let $q_{1}(m)$ and $q_{2}(m)$ be polynomials with rational coefficients of degree at most 2. A holomorphic triple $T=(E_{1}, E_{2},\phi)$ of type $(\overrightarrow{v}_{1},\overrightarrow{v}_{2})$ is called $(q_{1},q_{2})$-semistable (respectively, stable) if for any subsheaves $E'_{1}$ of $E_{1}$ and $E'_{2}$ of $E_{2}$ such that $0\neq E'_{1}\oplus E'_{2}\neq E_{1}\oplus E_{2}$ and $\phi(E'_{1})\subset E'_{2}$:
\begin{align}
&
rk(E_{2})q_{2}(m)\left(rk(E_{1})P_{E'_{1}}- rk(E'_{1}) P_{E_{1}}+rk(E'_{1})q_{1}(m)\right)\notag\\
&
+rk(E_{1})q_{1}(m)\left(rk(E_{2})P_{E'_{2}}- rk(E'_{2}) P_{E_{2}}-rk(E'_{2})q_{2}(m)\right)\leq0 / resp. <0.
\end{align}
\end{defi}
\begin{rmk}
Here, and throughout the article, when ever we compare two polynomials of same degree (in $m$), say $p$ and $p'$, we use the ordering rule that $p< p'$ if $p(m)<p'(m)$ for all $m\gg 0$.
\end{rmk}
\begin{rmk}\label{simplifyq}
Now assume $(\overrightarrow{v}_{1},\overrightarrow{v}_{2})$ is given such that $\overrightarrow{v}_{1}=(r_{1}\geq 1,\beta, n), \overrightarrow{v}_{2}=(r_{2}\geq 1,\gamma,m)$, $\beta,\gamma\in H^{2}(S,\mathbb{Z})$ are effective curve classes, and $n,m\in \mathbb{Z}$. Let $q_{1}(m)=q_{2}(m):=q(m)$ a polynomial of degree at most 2 with positive leading coefficient. Then, we obtain a simplified version of the above inequality:

\begin{align}\label{simplify}
&
r_{2} \left(r_{1}\cdot P_{E'_{1}}- rk(E'_{1}) P_{E_{1}}+rk(E'_{1})q(m)\right)\notag\\
&
\,\,\,\,\,\,\,\,\,\,\,\,\,\,\,\,\,\,\,\,\,\,\,\,\,\,\,\,\,\,\,\,\,\,\,\,\,\,\,\,\,\,\,\,\,\,\,\,\,\,\,\,\,\,+r_{1} \left(r_{2} \cdot P_{E'_{2}}- rk(E'_{2}) P_{E_{2}}-rk(E'_{2})q(m)\right)\leq0 / resp. <0.
\end{align}
\end{rmk}
Now based on \eqref{simplify} we state a consequence of our new notion of $q(m)$-semistability for ``\textit{torsion-free}" triples of type $(\overrightarrow{v}_{1},\overrightarrow{v}_{2})$. 
\begin{cor}(\textit{Conditions imposed on subsheaves of $E_{2}$})\label{defi13}
Let $q(m)$ be given by a polynomial as in Remark \ref{simplifyq}.  A holomorphic torsion-free triple $(E_{1}, E_{2},\phi)$ of fixed type  $(\overrightarrow{v}^{*}_{1}:=(r_{1},\beta, n), \overrightarrow{v}^{*}_{2}:=(r_{2},\gamma,m)), r_{2}, r_{1}\geq 1$, is $q(m)$-stable (for short $q$-stable) if and only if the following conditions are satisfied;
\begin{enumerate}
\item For all proper nonzero subsheaves $E'_{2}\subset E_{2}$ for which $\phi$ does not factor through $E'_{2}$ we have: $$r_{1}r_{2}P_{E'_{2}}< r_{1}rk(E'_{2}) P_{E_{2}}+r_{1}rk(E'_{2})q(m).$$
\item For all proper subsheaves, $E'_{2}\subset E_{2}$ for which the map $\phi$ factors through:
\begin{align*}
r_{2}\left(P_{E'_{2}}+q(m)\right)<rk(E'_{2})\left(P_{E_{2}}+q(m)\right)
\end{align*}
\end{enumerate}
\end{cor}
\begin{lem}\label{rk1}
Given a holomorphic triple, $(E_{1}, E_{2}, \phi)$ of type $(\overrightarrow{v}^{*}_{1},\overrightarrow{v}^{*}_{2})$ as in Remark  \ref{simplifyq}, the $q$-stability of the triple implies that the morphism $\phi$ is injective when $q(m)\gg 0$ is large enough (i.e. its leading positive coefficient is large enough). 
\end{lem}
\begin{proof}
Assume that $(E_{1}, E_{2}, \phi)$ is $q$-stable and $\phi$ is not injective. Now let $K:=\text{ker}(\phi)$, which induces the subtriple $K\rightarrow 0$ of $E_{1}\xrightarrow{\phi} E_{2}$ which clearly destabilizes $(E_{1},E_{2},\phi)$ since by substitution we see that:$$r_{2}\left(r_{1}P_{K}-r_{K}P_{E_{1}}+r_{K}q(m)\right)<0$$which does not hold true if $q(m)\gg 0$ is large enough, in particular the lower bound for such values of $q$ is given by$$r_{K}q(m) \geq \left(r_{K}P_{E_{1}}-r_{1}P_{K}\right).$$
\end{proof}
From now on, where ever needed, we will denote $q_{+_{\gg}}$ as values of $q$ satisfying conditions of Lemma \ref{rk1}.
\section{Moduli scheme of triples}\label{chap2}
\begin{defi}(\textit{Moduli stack of holomorphic triples})\label{T-def}
Define $\mathfrak{M}^{(\overrightarrow{v}^{*}_{1},\overrightarrow{v}^{*}_{2})}_{\text{T}}(S,q)$ to be the fibered category $\mathfrak{p}:\mathfrak{M}^{(\overrightarrow{v}^{*}_{1},\overrightarrow{v}^{*}_{2})}_{\text{T}}(S,q) \rightarrow Sch/\mathbb{C}$ such that for all $T\in Sch/\mathbb{C}$ the objects in $\mathfrak{M}^{(\overrightarrow{v}^{*}_{1},\overrightarrow{v}^{*}_{2})}_{\text{T}}(S,q)$ are $T$-flat families of $q$-stable holomorphic triples of type $(\overrightarrow{
v}^{*}_{1},\overrightarrow{v}^{*}_{2})$. Given a morphism of $\mathbb{C}$-schemes $g:T\rightarrow T'$ and two families of holomorphic triples $\Lambda_{T}:=(\mathscr{E}_{1},\mathscr{E}_{2},\phi)_{T}$ and $\Lambda_{T'}:=(\mathscr{E}_{1},\mathscr{E}_{2},\phi)_{T'}$ (sub-index indicates the base parameter scheme over which the family is constructed), a morphism $\Lambda_{T}\rightarrow \Lambda_{T'}$ in $\mathfrak{M}^{(\overrightarrow{v}^{*}_{1},\overrightarrow{v}^{*}_{2})}_{\text{T}}(S,q)$ is defined by an isomorphism: $$\nu_{T}: \Lambda_{T}\xrightarrow{\cong}(g\times \textbf{1}_{S})^{*}\Lambda_{T'}.$$
\end{defi}
\begin{prop}\label{stack}
Use definitions \ref{defi3} and \ref{T-def}. The fibered category $\mathfrak{M}^{(\overrightarrow{v}^{*}_{1},\overrightarrow{v}^{*}_{2})}_{\text{T}}(S,q)$ is a stack.
\end{prop}
\begin{proof} This is immediate from faithfully flat descent of coherent sheaves and morphisms of coherent sheaves \cite[Theorem 4.23]{a67}.
\end{proof}

 Next we have to show that there exists an algebraic moduli scheme which corepresents the moduli functor defining stack of holomorphic triples in Proposition \ref{stack};  The main requirement to construct such moduli scheme is the boundedness property for the family of triples of fixed given type. Schmitt \cite[Theorem 4.1]{S00}  studies the construction of the moduli space of quiver representations given as oriented trees $E_{1}\to E_{2}\to \cdots\to E_{n}$ over a variety $X$ such that $E_{i}\in \text{Coh}(X), 1\leq i\leq n$ are given as torsion free coherent sheaves. In fact the holomorphic triples defined in this article are special case of Schmitt's oriented trees for $n=2$. The author introduces the notion of Hilbert polynomial and stability condition for such quiver representations, where in fact he shows that our notion of $q$-semistability is a specialization of $\vartheta$-semistability criterion for special quiver representations given as holomorphic triples \cite[Example 1.3]{S00}. The author then proves that the set of isomorphism class of $\vartheta$-semistable representations ($q$-semistable holomorphic triples for us) with a fixed Hilbert polynomial is bounded \cite[Theorem 4.1]{S00}, which implies that for quivers of fixed type, the set of isomorphism classes of $\vartheta$-semistable representations satisfies a Castelnuvo-Mumford type $m$-regularity condition. The author then provides a construction of the parameter scheme of $\vartheta$-semistable representations and proves that a projective moduli scheme associated to this parameter scheme exists as a good quotient \cite[Theorem 4.6 and Theorem 4.7]{S00}. We denote this projective moduli scheme by $\mathscr{M}^{(\overrightarrow{v}^{*}_{1},\overrightarrow{v}^{*}_{2})}(S,q_{+_{\gg}})$ and assume that it comes with a universal triple, denoted by $(\mathscr{E}_{1}, \mathscr{E}_{2}, \Phi)$ \cite[Proposition 4.5]{S00}.
 
 \begin{rmk}
(Boundedness)  By construction of Schmitt \cite{S00} the $\vartheta$-semistable representations do satisfy the property that the set of isomoprhism classes of torsion-free coherent sheaves with fixed Hilbert polynomials sitting at each vertex of the quiver representation is bounded \cite[Theorem 4.1]{S00}. Similarly for us, and by our construction, the $q$-semistable  triples enjoy this property.
 \end{rmk}

 \section{More on Stability}\label{More}
We will now relate $q$-stability of holomorphic torsion-free triples to notions of Higgs stability of certain induced Higgs pairs on $S$. In order to do this, we need to introduce further notions of stability conditions and relations among them. We learned the approach in this section, which implements change of parameter of stability condition, from interesting discussions by Bradlow and Prada in \cite[Section 3]{BP96} \footnote{They work with slope of bundles on $X \times \mathbb{P}^{1}$ where $X$ is a Riemann surface, we work with Hilbert polynomial of coherent sheaves on complex surface $S$, however all the ideas in here are theirs.}.
\begin{defi}\label{q'}
($q'$-stability). Given a polynomial $q'$ of degree at most 2,  let $T':=(E'_{1}, E'_2, \phi')$ be a nontrivial holomorphic subtriple of $T:=(E_{1}, E_{2}, \phi)$. Define a function 
\begin{equation}\label{q'-ineq}
\Theta_{q'}(T')=\bigg(\frac{P_{E'_{1}}+P_{E'_{2}}}{r'_{1}+r'_{2}}-q'\bigg)-\frac{r'_{1}}{r'_{1}+r'_{2}}.\frac{r_{1}+r_{2}}{r_{1}}\bigg(\frac{P_{E_{1}}+P_{E_{2}}}{r_{1}+r_{2}}-q'\bigg)
\end{equation}
where $r_{i}=rk(E_{i}), r'_{i}=rk(E'_{i}), i=1,2$. The triple $T:=(E_{1}, E_{2}, \phi)$ is called $q'$-stable (resp. semistable) if $$\Theta_{q'}(T')<0(\leq 0)$$for all nontrivial subtriples $T':=(E'_{1}, E'_2, \phi')$. 
\end{defi}
\begin{prop}\label{qq'}
Consider \ref{defi13} and \ref{q'} and change of variable $$q'=\frac{P_{E_{2}}+q}{r_{2}},$$then any triple $T:=(E_{1}, E_{2}, \phi)$ is $q'$-stable if and only if it is $q$-stable. 
\end{prop}
\begin{proof}
Assume that $T:=(E_{1}, E_{2}, \phi)$ is $q'$-stable. Then by substitution of the value of $q'$ in the statement of the lemma, we have that for all nontrivial subtriples  $T':=(E'_{1}, E'_2, \phi')$
\begin{align*}
\left(\frac{P_{E'_{1}}+P_{E'_{2}}}{r'_{1}+r'_{2}}-\frac{P_{E_{2}}+q}{r_{2}}\right)-\frac{r'_{1}}{r'_{1}+r'_{2}}.\frac{r_{1}+r_{2}}{r_{1}}\left(\frac{P_{E_{1}}+P_{E_{2}}}{r_{1}+r_{2}}-\frac{P_{E_{2}}+q}{r_{2}}\right)<0
\end{align*} 
which after multiplying by $r_{1}r_{2}(r'_{1}+r'_{2})>0$ implies that
\begin{align*}
&r_{1}r_{2}P_{E'_{1}}+r_{1}r_{2}P_{E'_{2}}-r_{2}P_{E_{1}}(r'_{1}+r'_{2})+r_{2}q(r'_{1}+r'_{2})-r_{1}r'_{2}(P_{E_{2}}+P_{E_{2}})\notag\\
&
\,\,\,\,\,\,\,\,\,\,\,\,\,\,\,\,\,\,\,\,\,\,\,\,\,\,\,\,\,\,\,\,\,\,\,\,\,\,\,\,\,\,\,\,\,\,\,\,\,\,\,\,\,\,\,\,\,\,\,\,\,\,\,\,\,\,\,\,\,\,\,\,\,\,\,\,\,\,\,\,\,\,\,\,\,\,\,\,\,\,\,\,\,\,\,\,\,\,\,\,\,\,\,\,\,\,\,\,+r'_{2}(r_{1}+r_{2})P_{E_{1}}-r'_{2}(r_{1}+r_{2})q<0.
\end{align*}
The latter inequality, after simplification, implies
\begin{align*}
&
r_{2} \left(r_{1}\cdot P_{E'_{1}}- r'_{1} P_{E_{1}}+r'_{1}q(m)\right)+r_{1} \left(r_{2} \cdot P_{E'_{2}}- r'_{2} P_{E_{2}}-r'_{2}q(m)\right)<0,\notag\\
\end{align*}
which is the criteria for $q$-stability.
\end{proof}

\begin{defi}\label{sigma}
($\sigma$-stabillity) With $\sigma$ a polynomial define the $\sigma$-polynomial and $\sigma$-slope of a given triple $T:=(E_{1}, E_{2}, \phi)$ respectively by $$P_{\sigma}(T):=P_{E_{1}}+P_{E_{2}}+r_{1}\sigma$$
and$$\mu_{\sigma}(T):=\frac{P_{\sigma}}{r_{1}+r_{2}}.$$The triple $T=(E_{1}, E_{2}, \phi)$ is $\sigma$-stable if for all nontrivial subtriples $T':=(E'_{1}, E'_{2}, \phi')$ we have$$\mu_{\sigma}(T')<\mu_{\sigma}(T).$$
\end{defi}
\begin{prop}\label{qs}
Consider definitions \ref{q'} and \ref{sigma} and fix $\sigma$ and $q'$ such that
\begin{equation}\label{q'-sigma}
\sigma=\frac{r_{1}+r_{2}}{r_{1}}(q'-\frac{P_{E_{1}}+P_{E_{2}}}{r_{1}+r_{2}}).
\end{equation}
Then any triple $T=(E_{1}, E_{2}, \phi)$ is $q'$-stable if an only if it is $\sigma$-stable. 
\end{prop}
\begin{proof}
Assume that the triple $T$ is $\sigma$-stable, then for any nontrivial subtriple $T'$ we have $$\mu_{\sigma}(T'):=\frac{P_{E'_{1}}+P_{E'_{2}}+r'_{1}\sigma}{r'_{1}+r'_{2}}<\mu_{\sigma}(T):=\frac{P_{E_{1}}+P_{E_{2}}+r_{1}\sigma}{r_{1}+r_{2}}$$which after substitution for $\sigma$ induces the inequality
\begin{align*}
\frac{P_{E'_{1}}+P_{E'_{2}}}{r'_{1}+r'_{2}}+\frac{r'_{1}}{r'_{1}+r'_{2}}.\frac{r_{1}+r_{2}}{r_{1}}q'-\frac{r'_{1}(r_{1}+r_{2})}{r_{1}(r'_{1}+r'_{2})}.\frac{P_{E_{1}}+P_{E_{2}}}{r_{1}+r_{2}}<q'.
\end{align*}
which after re-writing implies$$\left(\frac{P_{E'_{1}}+P_{E'_{2}}}{r'_{1}+r'_{2}}-q'\right)-\frac{r'_{1}}{r'_{1}+r'_{2}}.\frac{r_{1}+r_{2}}{r_{1}}\left(\frac{P_{E_{1}}+P_{E_{2}}}{r_{1}+r_{2}}-q'\right)
<0.$$
\end{proof}
Propositions \ref{qq'} and \ref{qs} imply the following corollary
\begin{cor}\label{q-s}
There exists a set of suitable polynomials $q,\sigma$ (which relate to each other by  Proposition \ref{qs} and Proposition \ref{qq'}) such that any triple $T:=(E_{1}, E_{2}, \phi)$ is $q$-stable if and only if it is $\sigma$-stable. \qed 
\end{cor}
\begin{lem}
Suppose that $\phi=0$. The holomorphic triple $(E_{1}, E_{2}, 0)$ is $q'$-semistable if and only if $q'=\frac{P_{E_{2}}}{r_{2}}$ and $E_{1}, E_{2}$ are semistable sheaves. Such triple can not be $q'$-stable.
\end{lem}

\begin{proof}
Since the map $\phi=0$, the holomorphic subtriples of $(E_{1},E_{2},0)$ are all of the form $(E'_{1},E'_{2}, 0)$ where $E'_{i}\subset E_{i}, i=1,2$ are proper holomorphic subsheaves. Now applying the criteria of $q'$ in Definition \ref{q'} to a subtriple $(0,0, E'_{2})$ we obtain that $$q'>\frac{P_{E'_{2}}}{r'_{2}}.$$On the other hand, applying the $q'$-stability criteria to a subtriple of the form $(E_{1}, E'_{2}, 0)$ implies that $$q'\leq \frac{P_{E_{2}}-P_{E'_{2}}}{r_{2}-r'_{2}}=\frac{P_{E_{2}/E'_{2}}}{r_{E_{2}/E'_{2}}}.$$Hence, combining the above two inequalities we obtain  that $$\frac{P_{E'_{2}}}{r'_{2}} < q'\leq \frac{P_{E_{2}/E'_{2}}}{r_{E_{2}/E'_{2}}},$$which is  possible if and only if $E_{2}$ is semistable. Now applying the same procedure to subtriples of the form $(E'_{1},0,0)$ and $(E'_{1}, E_{2}, 0)$, respectively, implies that $E_{1}$ is semistable. Moreover note that, the obtained inequalities can not be made strict without reaching a contradiction. Hence the triple is stable for this choice of $q'$ if and only if $\phi \neq 0$.
\end{proof}

\begin{cor}
The map $\phi$ can not be zero for a $q'$-stable triple.
\end{cor}
Now we obtain some upper and lower bounds for allowable $q'$ and equivalently $\sigma$.
\begin{lem}
(Lower bound for $\sigma$). If $(E_{1},E_{2}, \phi)$ is a $\sigma$-stable triple. Then $$\sigma\geq \frac{P_{E_{1}}}{r_{1}}-\frac{P_{E_{2}}}{r_{2}}.$$

\begin{proof}
The lower bound for $\sigma$ can be obtained by applying $\sigma$-stability to subtriples $(0, E_{2}, \phi)$. We obtain that 
$$\frac{P_{E_{2}}}{r_{2}}\leq \frac{P_{E_{1}}+P_{E_{2}}+r_{1}\sigma}{r_{1}+r_{2}},$$which implies the inequality in the statement of the lemma. In particular if $E_{1}$ and $E_{2}$ have same reduced Hilbert polynomials, then the lower bound for $\sigma$ is $0$.
\end{proof}
\end{lem}
\begin{lem}\label{upperbound}
(Upper bound for $\sigma$). Let $(E_{1}, E_{2}, \phi)$ be a $\sigma$-stable triple. Then $$\sigma< \frac{P_{E_{1}}(r_{1}-2r_{2})+P_{E_{1}}(2r_{1}+r_{2})}{r_{1}}.$$
\end{lem}
\begin{proof}
The proof is essentially given by adapting an approach similar to \cite[Proposition 3.14]{BP96} to our situation. We work with $q'$-stability in \eqref{q'-ineq}, and later change to $\sigma$-stability, using Proposition \ref{qs}, to prove the claim. Let us denote $K:=\text{Ker}(\Phi), I:=\text{Im}(\phi)$ and consider the proper subtriples $T'_{1}:=(K,0, \phi)$ and $T'_{2}:=(E_{1}, I,\phi)$. Using Proposition \ref{qs}, since $(E_{1}, E_{2},\phi)$ is $q'$-stable, then we get the two inequalities

\begin{equation}\label{ineq1}
\theta_{q'}(T'_{1})=\left(\frac{P_{K}}{r_{K}}-q'\right)-\frac{r_{1}+r_{2}}{r_{1}}\bigg(\frac{P_{E_{1}}+P_{E_{2}}}{r_{1}+r_{2}}-q'\bigg)<0
\end{equation}
\begin{equation}\label{ineq2}
\theta_{q'}(T'_{2})=\bigg(\frac{P_{E_{1}}+P_{I}}{r_{1}+r_{I}}-q'\bigg)-\frac{r_{1}}{r_{1}+r_{I}}\cdot \frac{r_{1}+r_{2}}{r_{1}}\bigg(\frac{P_{E_{1}}+P_{E_{2}}}{r_{1}+r_{2}}-q'\bigg)<0
\end{equation}
which get simplified to
\begin{equation}\label{1}
r_{K}P_{K}-(P_{E_{1}}+P_{E_{2}})r_{K}+r_{K}r_{2}q'<0
\end{equation}

\begin{equation}\label{2}
P_{I}-P_{E_{2}}+q'(r_{2}-r_{I})<0
\end{equation}
 Now multiplying \eqref{1} by $r_{1}$, and adding to \eqref{2}, multiplied by $r_{1}r_{K}$, and noting that $P_{K}+P_{I}=P_{E_{1}}$ and $r_{K}+r_{I}=r_{1}$, we obtain $$q'<\frac{2P_{E_{2}}}{2r_{2}-r_1},$$which implies that. Now applying \eqref{q'-sigma}, and rewriting in terms of $\sigma$, we obtain the proof of the inequality as claimed in the statement of lemma. 
 \end{proof}
\begin{cor}
Using Lemma \ref{upperbound} and Proposition \ref{qq'} we obtain an upper bound for $q_{+_{\gg}}$:$$q_{+_{\gg}}<\frac{r_{1}P_{E_{2}}}{2r_{2}-r_{1}}.$$
\end{cor}

Now let us relate the notion of $\sigma$-stability defined above, to a notion of Higgs stability. 
\begin{lem}\label{l-higgs1}
Consider an $\mathscr{L}$-\textit{twisted split} Higgs pair 
\begin{align}\label{l-higgs}
\bigg(E_{1}\oplus E_{2}\bigg)\xrightarrow{\psi:=\bigg(\begin{array}{cc}0 & 0\\ \phi & 0\end{array}\bigg)}\bigg(E_{1}\oplus E_{2}\bigg)\otimes \mathscr{L}
\end{align}
where $\mathscr{L}$ is a line bundle on $S$. Then there exists a suitable choice of polynomial $q$  for which the induced $\mathscr{L}$-\textbf{twisted} torsion-free holomorphic triple $E_{1}\xrightarrow{\phi} E_{2}\otimes \mathscr{L}$ is $q$-stable if and only if the split Higgs pair in \eqref{l-higgs} is Higgs-stable.
\end{lem}
 \begin{proof}
 First we recall the notion of Higgs stability criteria. Given an ``$\mathscr{L}$-Higgs" pair $E\xrightarrow{\psi} E\otimes \mathscr{L}$, consisting of a torsion-free sheaf $E$ and a line bundle $\mathscr{L}$ on the surface, the pair is said to be stable with respect to an ample line bundle $\mathscr{O}_{S}(1)$ if and only if, $$\frac{P(E'(m))}{\text{rank}(E')}< \frac{P(E(m))}{\text{rank}(E)}$$ for any proper subsheaf, $E'$ of $E$, which is invariant under the Higgs map $\psi$. Here we are discussing a special case (the split case!) where $E=E_{1}\oplus E_{2}$. Firstly note that  a torsion-free sheaf $E'_{1}\oplus E'_{2}$ is a $\psi$-invariant subsheaf of $E_{1}\oplus E_{2}$ if and only if $T':E'_{1}\xrightarrow{\phi'}E'_{2}\otimes \mathscr{L}$ is a subtriple of $T: E_{1}\xrightarrow{\phi} E_{2}\otimes \mathscr{L}$.

 Now assume that the $\mathscr{L}$-Higgs pair in \eqref{l-higgs} is stable. Then for any $\psi$-invariant subsheaf $E'_{1}\oplus E'_{2}$ the following inequality is satisfied
 
 \begin{equation} \label{higgs-ineq}
 \frac{P_{E'_{1}}+P_{E'_{2}}}{r'_{1}+r'_{2}}< \frac{P_{E_{1}}+P_{E_{2}}}{r_{1}+r_{2}} 
\end{equation}
 
On the other hand note that 
$$\mu_{\sigma}(T)=\frac{P_{E_{1}}+P_{E_{2}\otimes \cL}+r_{1}\sigma}{r_{1}+r_{2}},$$
which is equal to

\begin{align}\label{ineq}
\mu_{\sigma}(T)=\frac{P_{E_{1}}+P_{E_{2}}}{r_{1}+r_{2}}+\frac{r_{1}\sigma}{r_{1}+r_{2}}+\frac{P'_{E_{2},\cL}}{r_{1}+r_{2}},
\end{align}
where $P'_{E_{2},\cL}:=P_{E_{2}\otimes \cL}-P_{E_{2}}=r_{2}\left(m(H\cdot c_{1}(\cL))+\frac{c_{1}(\cL)^2}{2}\right)$. Now let
\begin{equation}\label{sigmacrit}
\sigma=m(H\cdot c_{1}(\cL))+\frac{c_{1}(\cL)^2}{2}.
\end{equation}
We denote this choice by $\sigma_{crit}$ and observe that 
\begin{equation}\label{ineq-crit}
\mu_{\sigma_{crit}}(T)=\frac{P_{E_{1}}+P_{E_{2}}}{r_{1}+r_{2}}+Q_{\cL},
\end{equation} 
where $Q_{\cL}$ is the right hand side of \eqref{sigmacrit}. Note that $\sigma_{crit}$ is still given by a polynomial of degree with positive leading coefficient on $S$. Now consider any subtriple $T'\subset T$. We can immediately see that with the choice of $\sigma_{crit}$ and inequality \eqref{ineq-crit} $\mu_{\sigma_{crit}}(T')<\mu_{\sigma_{crit}}(T)$ if and only if inequality \eqref{higgs-ineq} is satisfied. Now using Proposition \ref{qs} we can immediately see the choice for $\sigma_{crit}$ implies $$q'_{crit} = \frac{P_{E_{1}}+P_{E_{2}}+r_{1}Q_{\cL}}{r_{1}+r_{2}},$$ for which Proposition \ref{qq'} implies then 

\begin{equation}\label{q-higgs}
q_{crit} = \frac{r_{2}\cdot P_{E_{1}}-r_{1}\cdot P_{E_{2}}}{r_{1}+r_{2}}+\frac{r_{1}r_{2}}{r_{1}+r_{2}}Q_{\cL},
\end{equation}
Hence for $q$ given as in \eqref{q-higgs}, the $\mathscr{L}$-twisted Higgs pair $(E_{1}\oplus E_{2}, \psi)$ is Higgs stable if and only if the $\mathscr{L}$-twisted triple $(E_{1}, E_{2}\otimes \mathscr{L}, \phi)$ is $q_{crit}$-stable.  
 \end{proof}

\begin{rmk}
Lemma \ref{l-higgs} is suggesting that changing the parameter of stability condition from, $q_{+_{\gg}}$ to  $q_{crit}$, one can relate the moduli space of trosion-free $q_{+_{\gg}}$-stable triples to the moduli space of Higgs-stable pairs as in Lemma \ref{l-higgs} by applying ``wallcrossing in the master space, developed by Motcizuki in \cite{M09}". We will elaborate on this procedure in the sequel to this article \cite{SY20} and here only discuss invariants of stable triples separately, first for $q_{crit}$-stability and then for $q_{+_{\gg}}$-stability.  
  \end{rmk}

 \section{$q_{crit}$-stable triples and Vafa-Witten invariants} \label{sec:threefold}
Start with $(S, h)$ given as a pair of a nonsingular projective surface $S$ with $H^1(\O_S)=0$, and $h:=c_1(\O_S(1))$, and let $$v:=(r,\gamma,m)\in H^{\text{ev}}(S, \mathbb{Q})=  H^0(S) \oplus H^2(S) \oplus H^4(S),$$ with $r\ge 1$.


Let $\cL$ be a line bundle on $S$ such that \begin{equation}\label{antican} H^0(\cL\otimes \omega_S^{-1})\neq 0, \end{equation} and let
\begin{align*}
 X:= \cL \xrightarrow{\q} S
\end{align*}
be the total space of $\mathscr{L}$ on $S$. 
Note that $X$ is  non-compact with canonical bundle $\omega_X\cong \q^*(\cL^{-1}\otimes \omega_S)$. In particular $X$ is a Calabi-Yau 3-fold if $\cL=\omega_S$.  Let $z:S\to X$ be the zero section inclusion.

\begin{notn} For simplicity, we use the symbols $z$ and $\q$ to indicate respectively the inclusion $z\times \id:S\times B\to X\times B$ and the projection $\q\times \id:X\times B\to S\times B$ for any scheme $B$. 
\end{notn}

The one dimensional complex torus $\mathbb{C}^{\ast}$
acts on $X$ by the multiplication on the fibers of $\q$, so that  \begin{equation}\label{piox} \q_* \O_X=\bigoplus_{i=0}^{\infty} \cL^{-i}\otimes \t^{-i},\end{equation} where $\t$ denotes the trivial line bundle on $S$ with the $\C^*$-action of weight 1 on the fibers.
Let $\Coh_c(X) \subset \Coh(X)$ be the abelian category of 
coherent sheaves on $X$
whose supports are compact. Given any such $\mathscr{O}_{X}$-module $\mathscr{E}$, via the map $\q_{*}$, one obtaines a $\q_{*}\O_X$-module $E$ on $S$ together with a co-action of $\mathscr{L}$ which defines the $\mathscr{L}$-Higgs pair on $S$, $E\xrightarrow{\phi}E\otimes \mathscr{L}.$ Moreover, given any $\mathscr{L}$-Higgs pair $(E,\phi)$, there is a co-action of $\q_* \O_X$ on $E$ which can be obtained by considering the co-actions of $\mathscr{L}^{i}, i>0$ on $E$, that is $$E\to E \otimes \mathscr{L}\to E \otimes \mathscr{L}^2\to \cdots.$$and summing over $i$, produces the $\q_{*}\O_{X}$ module we denote by $\q_{*}\mathscr{E}$. The reason for the latter notation is that, it is known by result of Serre that the push forward map $q_{*}$ induces an equivalence of categories between $\O_X$ modules $\mathscr{E}$ and $\q_{*}\O_X$ modules $\q_{*}\mathscr{E}$. Hence this implies an equivalence of categories between abelian category of $\mathscr{L}$-Higgs pairs, $E\xrightarrow{\phi}E\otimes \mathscr{L}$ on $S$, where $E=\q_{*}\mathscr{E}$, and the category of torsion (compactly supported) modules $\mathscr{E}$ obtained via Serre's equivalence. Now we define a stability condition for compactly supported $\O_X$-modules, $\mathscr{E}$.
The slope function $\mu_h$ on $\Coh_c(X) \setminus \{0\}$
\begin{align*}
\mu_h(\cE) =\frac{c_1(\q_{\ast}\cE) \cdot h}{\rank(\q_{\ast}\cE)} \in 
\mathbb{Q} \cup \{ \infty\}
\end{align*}
determines a slope stability condition on $\Coh_c(X)$\footnote{If $\rank(\q_{\ast}\cE)=0$, then $\mu_h(\cE)=\infty$.}. We state the following lemma without proof.
 
 \begin{lem}\label{TT-lem}
Under the equivalence of categories $\Coh_c(X)\cong \text{Higgs}_{\mathscr{L}}(S)$, $\mu_{h}$-stability of $\mathscr{E}\in \Coh_c(X)$ is equivalent to Higgs-stability of $(\q_{*}\mathscr{E}, \phi)$.
\end{lem}
\begin{proof}
This is proved by Tanaka-Thomas \cite[Lemma 2.9]{TT17a}.
\end{proof}
\subsection{Moduli space of compactly supported sheaves on $X$}\label{compact-moduli}
Let $\mMb_h(v)$ be the moduli space 
of $\mu_h$-stable sheaves $\cE \in \Coh_c(X)$ with $\ch(\q_{\ast}\cE)=v$, where$$v:=(r,\gamma,m)\in H^{\text{ev}}(S, \mathbb{Q})=  H^0(S) \oplus H^2(S) \oplus H^4(S),$$ with $r\ge 1$.  For simplicity, we also assume $\mMb_h(v)$ admits a universal family, denoted by $\eEb$. This is  the case if for example $\gcd(r,\gamma\cdot h)=1$ (see \cite[Corollary 4.6.7]{HL10}).
We  denote by $\bp$ the projection 
from $X \times \mMb_h(v)$ 
to $\mMb_h(v)$. By the condition \eqref{antican} and \cite{T98, HT10}, one obtains a natural perfect (threefold) obstruction theory on $\mMb_h(v)$ given by the complex\footnote{$\Eb$ is symmetric \cite{B09} if $\cL=\omega_S$.} 
$$\Eb:=\left(\tau^{[1,2]}\dR \hom_{ \bp}(\overline{\eE}, 
\overline{\eE})[1]
\right)^\vee
.$$ In fact Serre duality and Hirzerbruch-Riemann-Roch hold for the compactly supported coherent sheaves, even though $X$ is not compact. Therefore, using the latter we can calculate the rank of $\Eb$:
\begin{equation} \label{rankE}
\rank(\Eb)=\begin{cases} 0 & \cL=\omega_S,\\ r^2 c_1(\cL)\cdot(c_1(\cL)-\omega_S)/2+1 & \cL\neq \omega_S.\end{cases}
\end{equation}
By the localization technique developed by Graber-Pandharipande \cite{GP99},
we obtain the  $\mathbb{C}^{\ast}$-fixed 
perfect obstruction theory
\begin{align}\label{3fold}
\Ebf=\left(\tau^{[1,2]}\left( \dR \hom_{ \bp}(\overline{\eE}, 
\overline{\eE})[1]
\right)^\vee\right)^{\C^{\ast}}
\end{align} over the fixed locus $ \mMb_h(v)^{\mathbb{C}^{\ast}}$. Since the $\C^*$-fixed set of $X$ (i.e. $S$) is projective, we conclude that $\mMb_h(v)^{\mathbb{C}^{\ast}}$ is projective, therefore $E^{\bullet,\fix}$ gives the virtual fundamental class $[\mMb_h(v)^{\mathbb{C}^{\ast}}]^{\vir}$. Define

\begin{equation}\label{DThv3} \widehat{\DT^{\cL}_h}(v;\alpha)=\int_{[\mMb_h(v)^{\mathbb{C}^{\ast}}]^{\rm{vir}}}
\frac{\alpha}{e((E^{\bullet,\mov})^\vee)} \quad \quad \alpha \in H^{*}_{\C^*}(\mMb_h(v),\mathbb{Q})_{\s},\end{equation} where $E^{\bullet,\mov}$ is the $\C^*$-moving part of $\Eb$, and $e(-)$ indicates the equivariant Euler class.

\begin{rmk} \label{notdef}
 Note that $(E^{\bullet,\mov})^\vee$ is the virtual normal bundle of $\mMb_h(v)^{\mathbb{C}^{\ast}}$. If $\mMb_h(v)$ is compact then $\displaystyle \widehat{\DT^{\cL}_h}(v;\alpha)$ will be equal to $\int_{[\mMb_h(v)]^{vir}}\alpha$ via the virtual localization formula \cite{GP99}. This is the case when $c_1(\cL)\cdot h< 0$, as then one can see that all the stable sheaves must be supported scheme theoretically on the zero section of $\q:X\to S$. Note that if $c_1(\cL)\cdot h\ge 0$, then $\int_{[\mMb_h(v)]^{vir}}\alpha$ is not defined in general.
\end{rmk}


Now when $p_g(S)>0$, the fixed part of the obstruction theory $\Eb$ contains a trivial factor which causes the invariants $\widehat{\DT^{\cL}_h}(v)$ to vanish; we need then to reduce the obstruction theory $\Eb$ by a procedure elaborated in \cite{GSY17b} in full detail. However, to keep the self-containment of this article we briefly go over that construction here, and review some of the important facts about it. 

Define $C^\bullet$ to be obtained by applying $\dR \bp_*$ to the cone of the composition 
\begin{align*}
\q_{\ast}\dR \hom_{X\times \mMb_h(v)}
(\eEb, \eEb) \xrightarrow{\q_*}
\dR \hom_{S\times \mMb_h(v)}(\q_{\ast} \overline{\eE},
\q_{\ast}\overline{\eE}) \xrightarrow{\tr} 
\O_{S\times \mMb_h(v) },
\end{align*}  Note that $\q$ is an affine morphism and hence $R^i\q_*=0$ for $i>0$, and the first map $\q_*$ in the sequence  above is the natural map \cite[Proposition II.5.5]{H66}. Then, we define the reduced deformation-obstruction complex
\begin{equation} \label{3fold}
\Eb_{\red}:=\left(\tau^{\le 1}(C^\bullet )
\right)^\vee.
\end{equation}

\begin{lem} \label{perfect}
$\tau^{\le 1}(C^\bullet )$ is of perfect amplitude contained in $[0,1]$. Moreover,
$$h^0(\tau^{\le 1}(C^\bullet))\cong \ext^1_{\bp}(\eEb,\eEb),$$ and $h^1(\tau^{\le 1}(C^\bullet))$ fits into the short exact sequence $$0\to h^1(\tau^{\le 1}(C^\bullet))\to \ext^2_{\bp}(\eEb,\eEb) \to \O_{\mMb_h(v)}^{p_g}\to 0.$$
\end{lem}
\begin{proof} This is proved in \cite[Lemma 2.4]{GSY17b}.
\end{proof}
The following theorem states that $\Eb_{\red}$ induces a perfect deformation obstruction theory over $\mMb_h(v)$. 

\begin{thm} (\cite[Theorem 2.5]{GSY17b})\label{prop:red}
$\left(\Eb_{\red}\right)^{\vee}=\tau^{\le 1}(C^\bullet )$ is the virtual tangent bundle of a perfect obstruction theory over $\mMb_h(v)$.
\end{thm}

\begin{proof} In \cite[Theorem 2.5]{GSY17b} the authors provided two separate proofs for this claim, one using the Li-Tian approach \cite{LT98} and the other using the Behrend-Fantechi \cite{BF97} approach. Here we review only the Behrend-Fantechi approach for completeness. By Lemma \ref{perfect} we know that $\Eb_{\red}$ is of perfect amplitude contained in $[-1,0]$. It suffices to show that there exists a map $\theta: \Eb_{\red}\to \LL_{\mMb_h(v)}$ in derived category  that defines an obstruction theory, that is $h^0(\theta)$ is an isomorphism and and $h^{-1}(\theta)$ is an epimorphism. As usual it suffices to work with the truncation $\tau^{\ge -1}$ of the cotangent complex and this is what we mean by $ \LL$ in this proof. Again we use the fact that 
the composition $$\O_{S\times \mMb_h(v) } \xrightarrow{\id} \q_{\ast}\dR \hom_{X\times \mMb_h(v)}
(\eEb, \eEb) \xrightarrow{\q_*}
\dR \hom_{S\times \mMb_h(v)}(\q_{\ast} \overline{\eE},
\q_{\ast}\overline{\eE}) \xrightarrow{\tr} 
\O_{S\times \mMb_h(v) }$$ is $r\cdot \id$. This implies that the composition $\tr\circ \q_*$ splits and hence, after applying $Rp_*$, we get the isomorphism 
\begin{equation} \label{splittt}
\dR \hom_{\bp}
(\eEb, \eEb)\cong C^\bullet[-1] \oplus \dR p_* \O_{S\times \mMb_h(v) } \cong C^\bullet[-1] \oplus \O_{ \mMb_h(v) } \oplus \O_{ \mMb_h(v)}^{p_g}[-2].
\end{equation} 
Applying truncation functors to both sides of this splitting it is easy to see that
\begin{equation} \label{splittt1}
\tau^{[1,2]}(\dR \hom_{\bp}
(\eEb, \eEb))\cong \tau^{\le 1}(C^\bullet)[-1] \oplus \O_{ \mMb_h(v)}^{p_g}[-2].
\end{equation} 
Now there is a map $\alpha: \big(\tau^{[1,2]}(\dR \hom_{\bp}(\eEb, \eEb))\big)^{\vee}[1] \to  \LL_{\mMb_h(v)}$ constructed in  \cite[(4.10)]{HT10} by means of the truncated Atiyah class $$A(\eEb) \in \Ext^1_{\mMb_h(v)\times X}(\eEb,\eEb\overset {\dL}{\otimes}  \LL_{\mMb_h(v)\times X})$$ and an application of the truncation functor $\tau^{[1,2]}$. This
together with the splitting \eqref{splittt1} gives a map $$\theta:\Eb_{\red}= \big(\tau^{\le 1}(C^\bullet)\big)^\vee \to \LL_{\mMb_h(v)}.$$
It remains to show that $\theta$ is an obstruction theory. For this we use the criterion in \cite[Theorem 4.5]{BF97} and the fact that it is already proven that $\alpha$ is an obstruction theory in the last part of \cite[Section 4.4]{HT10}.

The question of being an obstruction theory is local in nature, so let $B_0\subset B$ be a closed immersion of affine schemes over $\C$ with the ideal sheaf $I$ such that $I^2=0$, and let $\G_0$ be a sheaf on $B_0\times X$ flat over $B_0$ corresponding to a morphism $f:B_0\to  \mMb_h(v)$. Let $\bp:X\times B_0 \to B_0$ and $p:S\times B_0\to B_0$ be the obvious projections. We have the chain of morphisms $$f^* \Eb_{\red}\xrightarrow{f^*\theta} f^*\LL_{\mMb_h(v)}\to \LL_{B_0}.$$ The pullback of the Kodaira-Spencer class $\kappa(B_0/B)\in \Ext^1(\LL_{B_0}, I)$ via the second arrow gives the obstruction class $\varpi(f) \in \Ext^1(f^*\LL_{\mMb_h(v)},I)$ for extending the map $f$ to $B$.
Pulling back further via the first arrow we get $\theta^*\varpi(f) \in \Ext^1(f^*\Eb_{\red},I)$. By  \cite[Theorem 4.5]{BF97} we have to show that $\theta^*\varpi(f)=0$ if and only if $f$ can be extended to $B$, and in this case the extensions form a torsor over $\Hom(f^*\Eb_{\red},I)$.

Similarly pulling back $\varpi(f)$ via $f^*\alpha$ we get 
\begin{align*}\alpha^*\varpi(f) \in &\Ext^1\big(f^* \big(\tau^{[1,2]}(\dR \hom_{\bp}(\eEb, \eEb))\big)^{\vee}[1],I\big)\cong \\
 &\mathbb{H}^2\big(\tau^{[1,2]}(\dR \hom_{\bp}(f^*\eEb, f^*\eEb\otimes \bp^*I))\big)\cong\\ 
&\Ext^2(\G_0,\G_0\otimes \bp^*I),\end{align*} where $\mathbb{H}^2$ denotes the hypercohomology and the isomorphisms are established in \cite[Section 4.4]{HT10} using the collapse of the Leray spectral sequence (and here is where $B_0$ affine is needed!). Taking the hypercohomology from both sides of \eqref{splittt1} and using the identifications above, we see that $\theta^*\varpi(f)$ is the $(\tr\circ \q_*)$-free part of $\alpha^*\varpi(f)$ (i.e. the part corresponding to the first summand in decomposition \eqref{splittt1}). But by Lemma \ref{tech2} below $\alpha^*\varpi(f)$ is the same as its own $(\tr\circ \q_*)$-free part, therefore $\theta^*\varpi(f)=\alpha^*\varpi(f)$. 

\begin{lem} \label{tech2}
$\q_*(\alpha^*\varpi(f))\in  \Ext^2(\q_*\G_0,\q_*\G_0\otimes p^*
I)_0.$ \end{lem}
\begin{proof} Define $$X_{B_0}:=X\times B_0,\quad  X_{B}:=X\times B,\quad S_{B_0}:=S\times B_0,\quad S_{B}:=S\times B.$$ 
Let $i_{X}: X_{B_0} \hookrightarrow X_{B_0}\times X_{B_0}$ and $i_{S}: S_{B_0} \hookrightarrow S_{B_0}\times S_{B_0}$ be the diagonal embeddings,  and $j_X: X_{B_0}\times X_{B_0} \hookrightarrow X_{B_0}\times X_B$ and $j_S: S_{B_0}\times S_{B_0} \hookrightarrow S_{B_0}\times S_B$ be the natural inclusions. Then define $$H_X:=\dL j_X^*\; j_{X*} \O_{\Delta_{ X_{B_0}}},\qquad H_S:=\dL j_S^*\; j_{S*} \O_{\Delta_{ S_{B_0}}}.$$ 
Using  the cartesian diagram 
$$\xymatrix{X_{B_0}\times X_{B_0} \ar[d]^-{\widetilde{\q}} \ar@{^(->}[r]^-{j_X} & X_{B_0}\times X_{B} \ar[d]^-{\widetilde{\q}}\\ 
S_{B_0}\times S_{B_0}  \ar@{^(->}[r]^-{j_S} & S_{B_0}\times S_{B},}$$ where $\widetilde{\q}:=(\q,\q)$, the fact $\O_{\Delta_{X_0}} =\widetilde{\q}^*\O_{\Delta_{S_0}}$, and flatness of $\q$ and hence $\widetilde{\q}$, we get  
\begin{equation} \label{qtilde} H_X=\widetilde{\q}^* H_S. \end{equation}
Huybrechts and Thomas define the universal obstruction class (\cite[Definition 2.8]{HT10}) $$\varpi_X:=\varpi(X_B/X_{B_0})\in \Ext^2_{X_{B_0}\times X_{B_0}}(\O_{\Delta_{ X_{B_0}}}, i_{X*}(\bp^*I))$$ as given by the extension class of the exact triangle $$ i_{X*}(\bp^*I)[1]\cong h^{-1}(H_X)[1]\to \tau^{\ge -1}(H_X)\to h^0(H_X)\cong \O_{\Delta_{ X_{B_0}}},$$ in which the first isomorphism is established in \cite[Lemma 2.2]{HT10} and the second isomorphism is given by the adjunction. The  universal obstruction class $$\varpi_S:=\varpi(S_B/S_{B_0})\in \Ext^2_{S_{B_0}\times S_{B_0}}(\O_{\Delta_{ S_{B_0}}}, i_{S*}(p^*I))$$ is defined similarly by using $H_S$ instead of $H_X$. By \eqref{qtilde} we have
\begin{equation}\label{varpiq}  \varpi_X=\widetilde{\q}^* \varpi_S.\end{equation}
Thinking of $\varpi_X$ and $\varpi_S$ as Fourier-Mukai kernels, and tensoring them respectively with pullbacks of $\G_0$ and $\q_*\G_0$, by  \cite[Thm 2.9, Cor 3.4]{HT10} we obtain the obstruction classes  
$$\varpi_X(\G_0)\in \Ext^2_X(\G_0,\G_0\otimes \bp^*I),\quad \varpi_S(\q_* \G_0) \in \Ext^2_S(\q_*\G_0,\q_*\G_0\otimes p^* I).$$ for deforming these sheaves. By \eqref{varpiq} and the commutative diagram 
$$\xymatrix{X_{B_0} \ar[d]^-{\q}& \ar[l]_-{\pr_1} X_{B_0}\times X_{B_0} \ar[d]^-{\widetilde{\q}} \ar[r]^-{\pr_2} & X_{B_0} \ar[d]^-{\q}\\ 
S_{B_0} & \ar[l]_-{\pr_1} S_{B_0}\times S_{B_0}  \ar[r]^-{\pr_2} & S_{B_0},}$$ where $\pr_1, \pr_2$ are obvious projections to the 1st and 2nd factors, an  application of projection formula gives (here for simplicity $\pr_{i*}$ denotes the derived pushforward)
\begin{align}\label{XBSB}\q_* \varpi_X(\G_0)&= \q_* \pr_{2*}(\pr_1^*\G_0 \otimes \varpi_X)\\ \notag
&=\q_*\pr_{2*}(\pr_1^*\G_0 \otimes \widetilde{\q}^*\varpi_S)=\pr_{2*}\widetilde{\q}_*(\pr_1^*\G_0 \otimes \widetilde{q}^*\varpi_S)\\ \notag
&=\pr_{2*}(\widetilde{\q}_*\pr_1^*\G_0 \otimes \varpi_S)  =\pr_{2*}(\pr_1^* \q_*\G_0 \otimes \varpi_S)=\varpi_S(\q_*\G_0). \end{align}

But $\varpi_X(\G_0)=\alpha^*\varpi(f)$ by \cite[Cor 3.4]{HT10} and \cite[Thm 4.5]{BF97} as we already know that $\alpha$ is an obstruction theory. So by \eqref{XBSB} we get 
\begin{equation}\label{varpiqG0} \q_*(\alpha^*\varpi(f))=\varpi_S(\q_*\G_0).  \end{equation}

 By \cite[Theorem 3.23]{T98}, 
the obstruction for deforming the line bundle $\det (q_* \G_0)$ is given by the trace of the obstruction class: \begin{equation} \label{tracezero} \tr(\varpi_S(\q_*\G_0)).\end{equation} However, there are no obstructions for deforming line bundles, and therefore \eqref{tracezero} vanishes, or equivalently $$\varpi_S(\q_*\G_0)\in  \Ext^2_S(\q_*\G_0,\q_*\G_0\otimes p^*
I)_0.$$ Now lemma follows from \eqref{varpiqG0}.

\end{proof}
Now back to proof of Theorem \ref{prop:red}: since $\alpha$ is an obstruction theory, by  \cite[Theorem 4.5]{BF97}, $\theta^*\varpi(f)=\alpha^*\varpi(f)=0$ if and only if $f$ can be extended to $B$, and in this case the extensions form a torsor over $$\Hom(f^* \big(\tau^{[1,2]}(\dR \hom_{\bp}(\eEb, \eEb))\big)^{\vee}[1],I)\cong\Hom(f^*\Eb_{\red},I),$$ where the isomorphism is again by applying hypercohomology to \eqref{splittt1}.
\end{proof}
Now by applying Theorem \ref{prop:red} to \eqref{3fold}, and using \cite{GP99},
we obtain the  $\mathbb{C}^{\ast}$-fixed 
perfect reduced obstruction theory
\begin{align*}
\Ebf_{\red}\to \LL_{\mMb_h(v)^{\mathbb{C}^{\ast}}}. 
\end{align*} over the fixed locus $ \mMb_h(v)^{\mathbb{C}^{\ast}}$.
Here, by construction  $$\rank(\Eb_{\red})=\rank(\Eb)+p_g(S).$$ In particular, when $p_g(S)=0$, we have $\Eb=\Eb_{\red}$. Moreover, the reduction that takes $\Eb$ to $\Eb_{\red}$ only affects the fixed parts of the virtual tangent bundles i.e. $$\Ebm=\Ebm_{\red}.$$

We now state some of earlier results about our reduced threefold obstruction theory without proof.


\begin{cor}\label{red-fixed}
$\Ebf_{\red}$ gives a perfect obstruction theory over $\mMb_h(v)^{\C^*}$, and hence a virtual fundamental class $$[\mMb_h(v)^{\C^*}]^{
\vir}_{\red}\in A_{*}(\mMb_h(v)^{\C^*}).$$
 \qed\end{cor}


In the rest of the paper, we will study the invariants defined below:

\begin{defi} \label{DThv4}
(\cite[Definition 2.10]{GSY17b}) We can define the DT invariants
\begin{align*}
\DT^{\cL}_h(v;\alpha)&:=
\int_{[\mMb_h(v)^{\C^*}]_{\red}^{\vir}}
\frac{\alpha}{e((E^{\bullet,\mov})^\vee)}\in \mathbb{Q}[\s,\s^{-1}] \quad \quad \alpha \in H^{*}_{\C^*}(\mMb_h(v),\mathbb{Q})_{\s}, 
\end{align*} 
Here $e(-)$ denotes the equivariant Euler class, $\s$ is the equivariant parameter, and $c(-)$ denotes the total Chern class. Note that $\Ebm=\Ebm_{\red}$.
\end{defi}

\begin{rmk}  \label{alpha1}
The invariant $\DT_h^{\cL}(v;\alpha)$ is the reduced version of the invariant $\widehat{\DT^{\cL}_h}(v;\alpha)$ given in \eqref{DThv3}. If $\alpha=1$ then it can be seen easily that 
$$\DT^{\cL}_h(v;1)\cdot \s^{\rank(\Eb_{\red})} \in \mathbb{Q},$$ where $\rank(\Eb_{\red})$ is given by \eqref{rankE} and \eqref{3fold}. In particular, if $\cL=\omega_S$, then $\rank(\Eb_{\red})=p_g(S)$.
\end{rmk}

\section{$\C^*$-fixed loci and identification with $q_{crit}$-stable triples on $S$}\label{fixed-loci}
We continue this section by giving a  description of the components of  $\mMb_h(v)^{\C^*}$. Suppose that $\cE$ is a closed point of $\mMb_h(v)^{\C^*}$. Since $\cE$ is a pure $\C^*$-equivariant sheaf, up to tensoring with a power of $\t$, we can assume that, for some partition $\lambda \vdash r$, with $\lambda=(\lambda_1,\cdots , \lambda_{\ell(\lambda)})$, we have $$\q_* \cE=\bigoplus_{i=0}^{\ell(\lambda)-1} E_{-i}\otimes \t^{-i},$$ where $E_{-i}$ is a rank $\lambda_{i+1}$ torsion free sheaf on $S$, and the $\O_X$-module structure on $\cE$ is given by a collection of nonzero maps of $\O_S$-modules (using \eqref{piox}):
$$\psi_i: E_{-i} \to E_{-i-1}\otimes \cL, \quad  \quad i=0,\dots, \ell(\lambda)-1.$$
Let $\cE_i:=z_*E_{-i}$, for any $i$ and let $\cE'_0:=\cE$. Define $\cE'_i$ for $i >0$ inductively by \begin{equation} \label{sescei} 0\to \cE'_{i+1}\otimes \t^{-1}\to \cE'_i \to \cE_i\to 0.\end{equation} Therefore, we get a filtration (forgetting the equivariant structures) $$\cE'_{\ell(\lambda)-1}\subset \cdots\subset \cE'_1\subset\cE'_0=\cE,$$ and the stability of $\cE$ imposes the following conditions:
\begin{equation}\label{stabs} \mu_h(\cE'_i)<\mu_h(\cE) \quad \quad i=1,\dots, \ell(\lambda).\end{equation}
Note that for all $j$ we have $$\q_*\cE'_j\otimes \t^{-j}=\bigoplus_{i=j}^{\ell(\lambda)-1} E_{-i}\otimes \t^{-i},$$ and hence \eqref{stabs} imposes some restrictions on the ranks and degrees of $E_{-i}$'s. This construction also works well for the $B$-points of the moduli space $\mMb_h(v)^{\C^*}$ for any $\C$-schemes $B$. As a result, one gets a decomposition of the $\mathbb{C}^{\ast}$-fixed locus 
$\mMb_h(v)^{\mathbb{C}^{\ast}}$ into connected components
$$\mMb_h(v)^{\C^*}=\coprod_{\lambda \vdash r}\mMb_h(v)^{\C^*}_{\lambda},$$ 
where in the level of the universal families
\begin{align}\label{decompos} &\q_*\left(\eEb|_{X\times \mMb_h(v)^{\C^*}_{\lambda}}\right)=\bigoplus_{i=0}^{\ell(\lambda)-1} \eE_{-i}\otimes \t^{-i},\\ \notag
&\Psi_i: \eE_{-i} \to \eE_{-i-1}\otimes \cL, \quad  \quad i=0,\dots, \ell(\lambda)-1,\\ \notag
& \eEb'_{\ell(\lambda)-1}\subset \cdots\subset \eEb'_1\subset\eEb'_0:=\eEb,
\end{align} 
in which $\eE_{-i}$ is a flat family\footnote{Since $\q$ is an affine morphism, $ \q_*\eEb$ is flat over $\mMb_h(v)^{\C^*}_{\lambda}$, and hence each weight space $\eE_{-i}$ is flat over $\mMb_h(v)^{\C^*}_{\lambda}$.} of
rank $\lambda_i$ torsion free sheaves on $S\times \mMb_h(v)^{\C^*}_{\lambda}$, $\Psi_i$ is a family of fiberwise nonzero maps over $\mMb_h(v)^{\C^*}_{\lambda}$, $\eEb_i:=z_* \eE_{-i}$,  and $\eEb'_i$ for $i >0$ are inductively defined by 

\begin{equation}\label{univeses} 0\to \eEb'_{i+1}\otimes \t^{-1}\to \eEb'_i \to \eEb_i\to 0.\end{equation}

In \cite{GSY17a, GSY17b} the authors studied the two extreme cases $\lambda=(r)$ and $\lambda=(1^r)$. By  construction, it is clear that the former case coincides set theoretically with the moduli space $\mM_h(v)$. Moreover, the authors also showed that the latter case is related to the $r$-fold nested Hilbert schemes (parameterizing $r$-fold nesting of twisted ideal sheaves of curves and points) on $S$.  In fact when $r=2$, these cases are the only possibilities, and hence, in $r=2$ case, the $\C^*$-fixed virtual class of $\mMb_h(v)$ decomposes as
\begin{prop} \label{virdec} (\cite[Proposition 3.1]{GSY17b}) Suppose that $r=2$, then
$$[\mMb_h(v)^{\C^*}]^{\vir}_{\red}=[\mMb_h(v)^{\C^*}_{(2)}]^{\vir}_{\red}+[\mMb_h(v)^{\C^*}_{(1^2)}]^{\vir}_{\red}$$  \qed \end{prop} 
\begin{rmk}\label{vanish}
Note that, when $S$ is a Fano complete intersection or a K3 surface $$\mMb_h(v)^{\C^*}_{(1^2)}\cong \varnothing.$$This is proven in \cite[Corollary 3.16]{GSY17a} and \cite[Section 7.2]{TT17a}. 
\end{rmk}
\subsection{Connections to Vafa-Witten theory}\label{connect-VW}
Let us assume now that $\cL= \omega_{S}$, the canonical bundle of $S$. In \cite{TT17a} Tanaka-Thomas defined and computed Vafa-Witten invariants by constructing a symmetric perfect obstruction theory over the moduli space of Higgs-stable of trace-free Higgs pairs $$\mathscr{M}^{s}_{\text{Higgs}}:\{\phi: E\to E\otimes \omega_{S}\mid tr(\phi)=0\}$$on $S$. \begin{rmk}
The authors \cite{TT17a} in fact showed that:
\begin{enumerate} 
\item $\mathscr{M}^{s}_{\text{Higgs}}\cong \mathscr{M}^{\omega_{S}}_h(v)$.
\item The $\C^*$-action on $\mathscr{M}^{s}_{\text{Higgs}}$ is the same as $\C^*$-action on $\mathscr{M}^{\omega_{S}}_h(v)$.
\item Over the $\C^*$-fixed locus of $\mathscr{M}^{s}_{\text{Higgs}}$, $tr(\phi)=0$ automatically, and hence  \begin{equation}\label{torus-isom}\mathscr{M}^{s, \C^*}_{\text{Higgs}}=\mathscr{M}^{\omega_{S}, \C^*}_h(v).\end{equation}
\end{enumerate}
On the other hand, as shown in \cite[Section 2.1]{GSY17b} the fixed part of symmetric obstruction theory in \cite{TT17a} is equivalent in $K$-theory to $\Ebf_{\red}$ in Corollary \ref{red-fixed} and the moving parts differ (in $K$-theory) by a trivial bundle of rank $p_{g}(S)$. Hence in fact$$[\mathscr{M}^{s, \C^*}_{\text{Higgs}}]^{vir}=[\mathscr{M}^{\omega_{S}, \C^*}_h(v)]^{vir}$$ and when $\alpha=1$ in Equation \eqref{DThv3}, the reduced DT invariants in Definition \ref{DThv4} coincide with Vafa-Witten invariants studied and defined mathematically by Tanaka and Thomas in \cite{TT17a}, up to scaling by an equivariant parameter:
\begin{equation}\label{contribution}
\DT^{\omega_S}_h(v;1)=\s^{-p_g}\VW_h(v).
\end{equation}
\end{rmk}

Now let us assume as in Proposition \ref{virdec} that $r=2$. Then by the isomorphism \eqref{torus-isom} and Lemma \ref{TT-lem} any $\cE\in \mathscr{M}^{\omega_{S}, \C^*}_{(1^2)}$ is realized by a $\C^*$-fixed trace-free Higgs stable pair $(q_{*}\cE, \phi)\in \mathscr{M}^{s, \C^*}_{\text{Higgs}}$ which then by Lemma \ref{l-higgs1} is equivalent to a $q_{crit}$-stable triple $E_{0}\to E_{-1}\otimes \omega_{S}$, where $E_{i}, i=0,1$ are torsion-free coherent sheaves of rank 1 on $S$ such that $q_{*}\cE=E_{0}\oplus E_{-1}\otimes \t^{-1}$. Hence by equality \eqref{contribution} the latter $q_{crit}$-stable triples (in cases where their moduli space is non-empty) contribute nontrivially to Vafa-Witten invariants defined by Tanaka-Thomas. In fact the authors in \cite{GSY17b} showed that via $\C^*$-virtual localization, such contributions of $q_{crit}$-stable triples in $\mathscr{M}^{\omega_{S}, \C^*}_{(1^2)}$ can be obtained as residue integrals against the virtual classes of nested Hilbert scheme of non-pure one dimensional subschemes of $S$, induced by a perfect deformaton-obstruction theory, constructed by the authors in \cite{GSY17a}. 

\subsection{Higher rank case}
Now, without loss of generality and to explain the main result in this article, suppose that $r>2$, then by the same analysis as above we obtain a  decomposition by partitions of $r$:
\begin{equation}\label{rk=3}
\mMw_h(v)^{\C^*}=\mMw_h(v)^{\C^*}_{(r)}\sqcup \mMw_h(v)^{\C^*}_{(1,r-1)}\sqcup\cdots \sqcup \mMw_h(v)^{\C^*}_{(1^r)},
\end{equation}

In this case, the first summand is set theoretically identified with moduli space of stable rank $r$ torsion-free sheaves on $S$ (these are known as instanton moduli spaces in physics). Moreover, the ``\textit{monopole branch}" is the union of remaining terms in \eqref{rk=3} where the last summand enjoys nontrivial contributions induced by nested Hilbert scheme (with length $r$ nesting of ideal sheaves of curves and points) on $S$ as shown in \cite{GSY17b}, and the remaining middle summands are the interesting objects of study for us in this article. Note that, similar to Remark \ref{vanish}, the monopole branch vanishes when $S$ is given as a Fano complete intersection or a K3 surface. In \cite{T18} Thomas considers the obstruction theory of moduli space of semistable Higgs pairs $(E,\phi)$, where $E=q_{*}\cE$, for $\cE$ a semistable compactly supported sheaf on $X:Tot(\omega_{S}\to S)$. Thomas then considers the moduli space of Joyce-Song pairs on $X$ given by maps $$s:\mathcal{O}(-n)\to \cE$$where $n\in \mathbb{N}$ is sufficiently large  so that $H^{i}(\cE(n))=0$ for all $i>0$. The semistability of $\cE$ ensures that Joyce-Song pairs are always stable with no nontrivial automorphisms, their moduli space is quasi-projective with a compact $\C^*$-fixed locus, and further he shows that the moduli space is equipped with a symmetric perfect deformation-obstruction theory governed by  $Rhom(I^{\bullet}, I^{\bullet})_{\perp}$ (\cite[Section 5]{T18}), where $I^{\bullet}=\{\mathcal{O}(-n)\to \cE\}$. The author then shows that over many connected components of the $\C^*$-fixed loci of the monopole branch the obstruction sheaf induced by symmetric perfect obstruction theory of pairs admits a nowhere vanishing cosection in the sense of \cite{KL13, MPT10}, which implies that their contribution to the Vafa-Witten invariants is zero. Here is the important theorem proved by the author:

\begin{thr}
\cite[Theorem 5.23]{T18} Consider the cosection $Ob_{\mathcal{P}}\to \mathcal{O}_{\mathcal{P}}$ on a connected component $\mathcal{P}\subset (\mathcal{P}^{\perp}_{\alpha,n})^{T}$ of the monopole branch.\\
At any closed point of the zero scheme $Z(\phi^\cdot)\subset \mathcal{P}$ the maps $\phi: E_{i}\rightarrow E_{i+1}$ are both injective and generically surjective on $S$ for each $i=0,\cdots, r-2$. In particular if $Z(\phi^\cdot)\neq \varnothing$ then $\text{rank}(E_{0})=\text{rank}(E_{1})=\cdots=\text{rank}(E_{r-1})$.
\end{thr}
An important corollary of author's theorem is the following corollary (as pointed by the author, also observed by Laarakker \cite{TL18} in his calculations in lower ranks).
\begin{cor}\label{multi-rank}
\cite[Corollary 5.30]{T18} If the multi-rank is non-constant,$$(\text{rank}(E_{0}), \text{rank}(E_{1}), \cdots,\text{rank}(E_{r-1})\neq (a,a,\cdots,a)\,\,\,\,\, \forall a \in \mathbb{N}$$ 
then $\mathcal{P}$ does not contribute to the refined invariants.
\end{cor}
Corollary \ref{multi-rank} then implies that for stable compactly supported sheaves $\cE$, with rank $r$ over their support, when $r$ is a prime number, the only nontrivial contributions to the Vafa-Witten invariants are given by invariants of $\mMw_h(v)^{\C^*}_{(r)}$ and $\mMw_h(v)^{\C^*}_{(1^{r})}$ components. As we discussed the invariants of $\mMw_h(v)^{\C^*}_{(1^{r})}$ can be written in terms of residue integrals against virtual cycles of $r$-fold nested Hilbert schemes, whose deformation-obstruction theory was constructed in full generality, for any $r$, by the authors in \cite{GSY17a}. Hence an interesting application of construction below to Vafa-Witten theory is to compute invariants of flags of torsion-free coherent sheaves $$E_{1}\xrightarrow{\phi_{1}} E_{2}\xrightarrow{\phi_{2}}\cdots\xrightarrow{\phi_{n-1}} E_{n}$$ where $(\text{rank}(E_{0}), \text{rank}(E_{1}), \cdots,\text{rank}(E_{n})= (a,a,\cdots,a)$ for $a\neq 1$. Based on Corollary \ref{multi-rank} the first important intance of such scenario is when $r=4$. In this case there is a decomposition of the fixed locus
\begin{equation}\label{rk=4}
\mMw_h(v)^{\C^*}=\mMw_h(v)^{\C^*}_{(4)}\sqcup \mMw_h(v)^{\C^*}_{(1,3)}\sqcup \mMw_h(v)^{\C^*}_{(1,1,2)}\sqcup \mMw_h(v)^{\C^*}_{(2,2)}\sqcup \mMw_h(v)^{\C^*}_{(1^4)},
\end{equation}
and by Corllary \ref{multi-rank}, other than first and the last summand, only the invariants of $(2,2)$-locus contribute nontrivially to the Vafa-Witten invariants. Despite the latter fact, and in what follows below, under some assumptions, we aim at providing a more generalized construction of perfect deformation-obstruction theory, not only for the component $\mMw_h(v)^{\C^*}_{(2,2)}$, but also for the remaining middle components $\mMw_h(v)^{\C^*}_{(1,3)}$ and $\mMw_h(v)^{\C^*}_{(1,1,2)}$ which do not contribute to Vafa-Witten theory. For our construction we do not need an assumption on the line bundle $\cL$, however these only relate to Vafa-Witten theory when $\cL=\omega_{S}$ and the flag has constant multi-rank. One main issue to construct an absolute perfect obstruction theory over moduli space of higher rank flags is that given a partition of $r$, say when $\lambda\neq (1)^r, (r)$, the $\C^*$-fixed points in $\mMw_h(v)^{\C^*}_{\lambda}$ are given by flags of torsion-free coherent sheaves (with non-constant multi-rank) on $S$ $$E_{1}\xrightarrow{\phi_{1}} E_{2}\xrightarrow{\phi_{2}}\cdots\xrightarrow{\phi_{n-1}} E_{n}$$where $\phi_{i}, i=1,\cdots, n$ are nonzero maps but not necessarily injective,\footnote{It is in fact easy to come up with examples showing this fact.} and as we will see in Theorem \ref{rel} this causes an issue to obtain a well-defined virtual cycle for such flag moduli spaces: as lack of injectivity causes nonvanishing cohomology sheaves induced by deformation-obstruction complex of flag moduli space which prevent it from being perfect of amplitude $[-1,0]$. On the other hand, as we elaborated in Lemma \ref{rk1} given a triple $(E_{1}, E_{2}, \phi)$ the morphism $\phi$ is injective if and only if $(E_{1}, E_{2}, \phi)$ is $q_{+_{\gg}}$-stable, while $q_{crit}$-stability will not induce such constraint on the map $\phi$. Therefore in what follows below, we consider the moduli space of flags obtained by gluing the moduli spaces of $q_{+_{\gg}}$-stable triples and construct a perfect obstruction theory for them. We then argue that, in order to compute the contribution of loci in Corollary \ref{multi-rank}, one needs to perturb the parameter $q$ and perform a wallcrossing in the master space, as developed by Mochizuki in \cite{M09} but modified for flag moduli spaces, to relate the invariants of $q_{+_{\gg}}$-stable flags to the invariants of $q_{crit}$-stable flags. Note that by Lemma \ref{l-higgs1} the $q_{crit}$-stable flags are the ones which are identified with Higgs-stable $\omega_{S}$-twisted Higgs pairs considered by Tanaka-Thomas in \cite{TT17a, TT17b} and Thomas in \cite{T18}, with nontrivial contribution to Vafa-Witten invariants. We elaborate on this wallcrossing machinery and computation of invariants in some manageable cases over both curves and surfaces in the sequel to this article \cite{SY20} and for now only define a well behaved virtual moduli cycle theory for $q_{+_{\gg}}$-stable triples.

\section{Deformation-obstruction theory of $q_{+_{\gg}}$-stable triples}\label{def-obs}

Let us discuss a special case of length 2 flags first, i.e. the triples. Consider the moduli space $\mathscr{M}^{(\overrightarrow{v}^{*}_{1},\overrightarrow{v}^{*}_{2})}(S,q_{+_{\gg}})$ corepresenting to the moduli functor in \eqref{rk=3}, parametrizing $q_{+_{\gg}}$-stable triples $E_{1}\xrightarrow{\phi} E_{2}$ on surface $S$, with type $(\overrightarrow{v}^{*}_{1}, \overrightarrow{v}^{*}_{2})$.  Let us denote by $$\pi': \mathscr{M}^{(\overrightarrow{v}^{*}_{1},\overrightarrow{v}^{*}_{2})}(S,q_{+_{\gg}})\to \mathscr{M}^{\overrightarrow{v}^{*}_{1}},$$ the natural forgetful map, which sends a given holomorphic triple $(E_{1}, E_{2}, \phi)$ to $E_{1}$. Let us assume that $\overrightarrow{v}^{*}_{1}:=(r_{1}, \beta_{1},n_{1})$ and $\overrightarrow{v}^{*}_{1}:=(r_{2}, \beta_{2},n_{2})$, with $r_{i}\geq 1, i=1,2$, and $\beta_{1},\beta_{2}$ effective curve classes. 

\begin{assum}\label{smoothness}
Let us assume for simplicity that $\overrightarrow{v}^{*}_{1}$ is given such that the moduli space $\mathscr{M}^{\overrightarrow{v}^{*}_{1}}$ is smooth. This is possible for instance when $E_{1}$ is given as an ideal sheaf of zero-dimensional subschemes of fixed length in $S$, or when $E_{1}$ is given as a torsion-free higher rank stable sheaf with $gcd(rk(E_{1}), deg(E_{1}))=1$, and $K_{S}\leq 0$.  
\end{assum}
In this section we first prove that there exists a perfect relative deformation-obstruction theory for the map $\pi'$ and later, show that an absolute perfect deformation-obstruction theory for $\mathscr{M}^{(\overrightarrow{v}^{*}_{1},\overrightarrow{v}^{*}_{2})}(S,q_{+_{\gg}})$ can be induced from this relative theory.

\begin{notn} In what follows, we will denote the universal triple over $\mathscr{M}^{(\overrightarrow{v}^{*}_{1},\overrightarrow{v}^{*}_{2})}(S,q_{+_{\gg}})$ by $\Phi:=\mathscr{E}_{1}\to \mathscr{E}_{2}$. Let $\pi:S\times \mathscr{M}^{(\overrightarrow{v}^{*}_{1},\overrightarrow{v}^{*}_{2})}(S,q_{+_{\gg}})\to \mathscr{M}^{(\overrightarrow{v}^{*}_{1},\overrightarrow{v}^{*}_{2})}(S,q_{+_{\gg}})$ be the projection. The morphism $\pi$ is a smooth morphism of relative dimension 2 and hence by Grothendieck-Verdier duality $\pi^!(-):=\pi^*(-)\otimes \omega_{\pi}[2]$ is a right adjoint of $\dR \pi_*$. This fact will be exploited frequently below. We denote the derived functor $\dR\pi_*\dR\hom$ by $\dR\hom_\pi$ and its $i$-th cohomology sheaf by $\ext^i_\pi$. 
\end{notn}
 Applying the functors $\dR\hom_\pi\left (-,\mathscr{E}_{2}\right)$ and $\dR\hom_\pi\left (\mathscr{E}_{1}, - \right)$ to the universal map $\Phi$, we get the following morphisms in the derived category 
 \begin{align} \label{morph1}&\dR\hom_\pi\left(\mathscr{E}_{2}, \mathscr{E}_{2}\right ) \xrightarrow{\Xi}\dR \hom_\pi\left(\mathscr{E}_{1}, \mathscr{E}_{2}\right) \\\notag &\dR\hom_\pi\left(\mathscr{E}_{1}, \mathscr{E}_{1}\right ) \xrightarrow{\Xi'} \dR \hom_\pi\left(\mathscr{E}_{1}, \mathscr{E}_{2}\right). \end{align}

Let $T$ be any scheme over the $\C$-scheme $U$, and let 
\begin{equation*}
\xymatrix@=1em{
T \ar[d]_-a \ar@{^{(}->}[r] & \Tb \ar[ld]^-{\overline{a}}\\
U} 
\end{equation*}  be a square zero extension over $U$ with the ideal $J$. As $J^2=0$, $J$ can be considered as an $\O_{T}$-module. Suppose we have the Cartesian diagram

\begin{equation}  \label{sqzero}
\xymatrix{
T \ar[d]^a \ar[r]^-g & \mathscr{M}^{(\overrightarrow{v}^{*}_{1},\overrightarrow{v}^{*}_{2})}(S,q_{+_{\gg}}) \ar[d]^{\pi'}\\
U \ar[r] & \mathscr{M}^{\overrightarrow{v}^{*}_{1}}.} 
\end{equation} The bottom row of the \eqref{sqzero} corresponds to the $U$-point $\mathscr{E}_{U}$ of $\mathscr{M}^{\overrightarrow{v}^{*}_{1}}$, and the top row corresponds to the $T$-point (a $T$-family of triples)
\begin{equation} \label{Tpoint} (\phi_T: \mathscr{E}_{1,T}\to \mathscr{E}_{2,T})\end{equation} of $\mathscr{M}^{(\overrightarrow{v}^{*}_{1},\overrightarrow{v}^{*}_{2})}(S,q_{+_{\gg}})$ in which $\mathscr{E}_{1,T}=(\id,a)^{*}(\mathscr{E}_{1,U})$. Let $\pi_T$ be the projection from $S\times T\to  T$. By \cite[Prop. IV.3.2.12]{Ill} and \cite[Thm 12.8]{JS12}, there exists an element $$\ob:=\ob( \phi_T,J) \in \Ext^2_{S\times T}\left(\coker(\phi_T), \pi^*_TJ \otimes \mathscr{E}_{2,T}\right),$$ whose vanishing is necessary and sufficient to extend the $T$-point \eqref{Tpoint} to a $\Tb$-point $$(\mathscr{E}_{1,\Tb}, \mathscr{E}_{2,\Tb}, \phi_{\Tb}) $$ of $\mathscr{M}^{(\overrightarrow{v}^{*}_{1},\overrightarrow{v}^{*}_{2})}(S,q_{+_{\gg}})$ such that $\mathscr{E}_{1,\Tb}=(\id, \overline{a})^{*}(\mathscr{E}_{1,U})$.  In fact, pulling back along the map $(\id, \overline{a})$ means that by \cite[Prop. IV.3.2.12]{Ill}, $\ob$ is the obstruction for deforming the morphism $\phi_T$ while the deformation $\mathscr{E}_{1,\Tb}$ of $\mathscr{E}_{1,T}$  is given. Suppose that $\phi_{\Tb}: \mathscr{E}_{1,\Tb} \to \F$ is such a deformation, where $ \F$ is a flat family of torsion free sheaves with $\F|_{S\times T}=\mathscr{E}_{2,T}$. 
If $\ob = 0$ then by \cite[Prop. IV.3.2.12]{Ill}, the set of isomorphism classes of such deformations forms a torsor under the group$$\Ext^1_{S\times T}\left(\coker(\phi_T), \pi_T^*J \otimes \mathscr{E}_{2,T}\right).$$ 
Furthermore, by \cite[Thm 12.9]{JS12}), the element $\ob$ is the cup product of Illusie's reduced Atiyah class  \begin{equation}\label{atiyah}\at_{\text{red}}( \phi_T)\in \Ext^1_{S\times T}\left(\coker(\phi_T),\pi^*_T\LL_{a}\otimes \mathscr{E}_{2,T}\right),\end{equation} which is a reduction of the Atiyah class $\at_{\O_{S\times T}/\O_{S\times U}}(\phi_{T})$ of $\phi_T$ \cite[IV. 2.3.]{Ill}, that only depends on the data $( \mathscr{E}_{2,T},\phi_T)$  and the pullback of 
\begin{equation}\label{cotsqzero} e(\Tb)\in \Ext^1_T(\LL_{a},J)\end{equation}  associated to the square zero extension $T\hookrightarrow \Tb$, twisted by  $\mathscr{E}_{2,T}$. The following Lemma states this result.

\begin{lem}
The obstruction class $\ob$ satisfies the identity $$\ob=\at_{\text{red}}(\phi_T)\cup (\pi_{T}^{*}e(\Tb)\otimes \id)$$where $$\pi_{T}^{*}e(\Tb)\in \Ext^1_{S\times T}(\pi_{T}^*\LL_{a}\otimes \mathscr{E}_{2,T},\pi_{T}^* J\otimes \mathscr{E}_{2,T})$$
\end{lem}
\begin{proof}
The proof follows exactly the same procedure in \cite[Thm 12.9]{JS12}, and depicted in their Diagram (12.17). The only slight change in the proof of \cite[Thm 12.9]{JS12}, is that in the diagram (12.17) of [ibid], the first vertical arrow must be replaced with $$\at_{\O_{S\times T}/\O_{S\times U}}(\phi_T):\coker(\phi_T)\to k^1\left(\mathbb{L}^{\bullet,\text{gr}}_{(\O_{S\times T}\oplus \mathscr{E}_{1,T})/\O_{S\times U}}\otimes (\O_{S\times T}\oplus \mathscr{E}_{2,T})\right)[1].$$Here the operator $k^1(-)$ takes the degree $1$ graded piece of a graded object. 
By flatness of $\mathscr{E}_{1,U}$ over $\O_U$, we see that $\O_{S\times U}\oplus \mathscr{E}_{1,U}$ is flat over $\O_U$ and therefore $$(\O_{S\times U}\oplus \mathscr{E}_{1,U})\overset{\dL}{\otimes}_{\O_{S\times U}} \O_{S\times T}\cong(\O_{S\times U}\oplus \mathscr{E}_{1,U})\otimes_{\O_{S\times U}} \O_{S\times T}\cong \O_{S\times T}\oplus \mathscr{E}_{1,T},$$ hence we can write (see \cite[II.2.2]{Ill}) $$\mathbb{L}^{\bullet,\text{gr}}_{(\O_{S\times T}\oplus \mathscr{E}_{1,T})/\O_{S\times U}}
\cong\left( \mathbb{L}^{\bullet}_{\O_{S\times T}/\O_{S\times U}}\otimes(\O_{S\times T}\oplus \mathscr{E}_{1,T})\right)\oplus \left(\mathbb{L}^{\bullet,\text{gr}}_{(\O_{S\times U}\oplus \mathscr{E}_{1,U})/\O_{S\times U}}\otimes(\O_{S\times T}\oplus \mathscr{E}_{1,T})\right),$$ and hence as in the proof of \cite[Thm 12.9]{JS12}, composing $\at_{\O_{S\times T}/\O_{S\times U}}(\phi_T)$ with the projection $$\mathbb{L}^{\bullet,\text{gr}}_{(\O_{S\times T}\oplus \mathscr{E}_{1,T})/\O_{S\times U}}\to\mathbb{L}^{\bullet}_{\O_{S\times T}/\O_{S\times U}}\cong \pi^*_T\LL_{a},$$ we arrive at the definition of $\at_{\text{red}}(\phi_T)$ in \eqref{atiyah}.  
\end{proof}

\begin{thm} \label{rel}
The complex $$\Fbr:=\cone(\Xi)^\vee\cong \dR\hom_\pi\left(\coker(\Phi),\mathscr{E}_{2}\right)^\vee[-1]$$ defines a relative perfect obstruction theory for the morphism $\pi':\mathscr{M}^{(\overrightarrow{v}^{*}_{1},\overrightarrow{v}^{*}_{2})}(S,q_{+_{\gg}}) \to \mathscr{M}^{\overrightarrow{v}^{*}_{1}}$. In other words, $\Fbr$ is perfect with amplitude $[-1,0]$, and there exists a morphism in  the derived category  $\alpha: \Fbr\to \LL_{\pi'}$, such that $h^0(\alpha)$ and $h^{-1}(\alpha)$ are respectively isomorphism and epimorphism.  \end{thm}
\begin{proof}
\textbf{Step 1:} \emph{(perfectness)} We show that the dual complex $\Fbrv$ is perfect with amplitude $[0,1]$. By the base change and the same argument as in the proof of \cite[Lemma 4.2]{HT10}, it suffices to show that $h^i(\dL t^* \Fbrv)=0$ for $i\neq 0, 1$, where $t:p\hookrightarrow \mathscr{M}^{(\overrightarrow{v}^{*}_{1},\overrightarrow{v}^{*}_{2})}(S,q_{+_{\gg}})$ is the inclusion of an arbitrary closed point $p=(E_{1},E_{2},\phi)\in \mathscr{M}^{(\overrightarrow{v}^{*}_{1},\overrightarrow{v}^{*}_{2})}(S,q_{+_{\gg}})$. Note that by the definition of the universal families, $\dL t^* \mathscr{E}_{1}=E_{1}$ and $\dL t^* \mathscr{E}_{2}=E_{2}$. Therefore, by the definition of $\Fbrv$ we get the exact sequence 
$$\dots \to \Ext^{i}_S(E_{1},E_{2}) \to h^i(\dL t^* \Fbrv) \to \Ext^{i+1}_S(E_{2},E_{2})\to \dots .$$
All the $\Ext^i_S$ for $i\neq -1, 0, 1, 2$ vanish, so we deduce  that $h^i(\dL t^* \Fbrv)=0$ for $i\neq -1, 0, 1,  2$. From the sequence above we see that $$h^{-1}(\dL t^* \Fbrv)=\ker\big(\Hom_S(E_{2},E_{2}) \to \Hom_S(E_{1},E_{2})\big).$$ 
We would like to claim that $h^{-1}(\dL t^* \Fbrv)\cong 0$. Suppose that $\Hom_S(E_{2},E_{2}) \to \Hom_S(E_{1},E_{2})$ is not injective. Then there exists a nonzero map $f: E_{2}\to E_{2}$ such that the composition $E_{1}\xrightarrow{\phi} E_{2}\xrightarrow{f} E_{2}$ is zero. Now let us denote the image and the kernel of the map $f$ by $\text{Im}(f), \text{ker}(f)$ respectively. Since $\text{ker}(f)\subset E_{2}$ through which the morphism $\phi$ factors, by $q$-stability, and using Definition \ref{defi8}, the following strict inequality is satisfied
\begin{equation}\label{first-ineq}
\frac{P_{\text{ker}(f)}}{rk(\text{ker}(f))}+\frac{q(m)}{rk(\text{ker}(f))}<\frac{P_{E_{2}}}{r_{2}}+\frac{q(m)}{r_{2}}. 
\end{equation}
On the other hand making the same argument for $\text{Im}(f)$ we obtain
\begin{equation}\label{second-ineq}
\frac{P_{\text{Im}(f)}}{rk(\text{Im}(f))}+\frac{q(m)}{rk(\text{Im}(f))}<\frac{P_{E_{2}}}{r_{2}}+\frac{q(m)}{r_{2}} 
\end{equation}
Now, substituting $P_{\text{ker}(f)}=P_{E_{2}}-P_{\text{Im}(f)}$ and $rk(\text{ker}(f))=r_{2}-rk(\text{Im}(f))$ in to \eqref{first-ineq} we obtain $$\frac{P_{E_{2}}}{r_{2}}+\frac{q(m)}{r_{2}}< \frac{P_{\text{Im}(f)}}{rk(\text{Im}(f))},$$which together with \eqref{second-ineq} provides us with $$\frac{P_{\text{Im}(f)}}{rk(\text{Im}(f))}+\frac{q(m)}{rk(\text{Im}(f))}<\frac{P_{E_{2}}}{r_{2}}+\frac{q(m)}{r_{2}}<\frac{P_{\text{Im}(f)}}{rk(\text{Im}(f))}.$$In particular we obtain $$\frac{P_{\text{Im}(f)}}{rk(\text{Im}(f))}+\frac{q(m)}{rk(\text{Im}(f))}<\frac{P_{\text{Im}(f)}}{rk(\text{Im}(f))}.$$which simplifies to$$\frac{q(m)}{rk(\text{Im}(f))}<0,$$which is contradictory with the assumption on $q(m)\gg 0$ for $q_{+_{\gg}}$-stable triples, hence the morphism $\Hom_S(E_{2},E_{2}) \to \Hom_S(E_{1},E_{2})$ is injective which implies $$h^{-1}(\dL t^* \Fbrv)\cong 0.$$ To prove $h^2(\dL t^* \Fbrv)=0$, we show that the map $$\Ext^{2}_S(E_{2},E_{2})\to \Ext^{2}_S(E_{1},E_{2})$$ in the exact sequence above is surjective, or equivalently by Serre duality, the dual map $$\Hom_S(E_{2},E_{1}\otimes \omega_S)\to \Hom_S(E_{2},E_{2}\otimes \omega_S)$$ is injective.  But by Lemma \ref{rk1}, this follows after applying the left exact functor $\Hom_S(E_{2},-)$ to the injection $E_{1}\otimes \omega_S\to E_{2}\otimes \omega_S$ that is induced  by tensoring the map $\phi$ above by $ \omega_S$.

\textbf{Step 2}: \emph{(map to the cotangent complex)} We construct a morphism in derived category $\alpha: \Fbr\to \LL_{\pi'}$. Consider the reduced Atiyah class \eqref{atiyah} for $T=\mathscr{M}^{(\overrightarrow{v}^{*}_{1},\overrightarrow{v}^{*}_{2})}(S,q_{+_{\gg}})$ and $U=\mathscr{M}^{\overrightarrow{v}^{*}_{1}}$. It defines an element in 

\begin{align*}
&\Ext^{1}_{S\times \mathscr{M}^{(\overrightarrow{v}^{*}_{1},\overrightarrow{v}^{*}_{2})}(S,q_{+_{\gg}})}\left(\coker(\Phi),  \pi^{*}\LL_{\pi'}\otimes \mathscr{E}_{2}\right)\cong\\
&
 \Ext^{1}_{S\times \mathscr{M}^{(\overrightarrow{v}^{*}_{1},\overrightarrow{v}^{*}_{2})}(S,q_{+_{\gg}})}\left(\dR \hom\left(\mathscr{E}_{2}, \coker(\Phi)\right), \pi^{*}\LL_{\pi'}\right)\cong \\
&
\Ext^{1}_{S\times \mathscr{M}^{(\overrightarrow{v}^{*}_{1},\overrightarrow{v}^{*}_{2})}(S,q_{+_{\gg}})}\left(\dR \hom \left(\mathscr{E}_{2}, \coker(\Phi)\otimes \omega_{\pi}[2]\right), \pi^{!}\LL_{\pi'}\right)\cong \\
&
 \Ext^{1}_{\mathscr{M}^{(\overrightarrow{v}^{*}_{1},\overrightarrow{v}^{*}_{2})}(S,q_{+_{\gg}})}\left(\dR \hom_\pi \left(\mathscr{E}_{2} , \coker(\Phi)\otimes \omega_{\pi}[2]\right), \LL_{\pi'}\right)\cong\\
&
\Hom_{\mathscr{M}^{(\overrightarrow{v}^{*}_{1},\overrightarrow{v}^{*}_{2})}(S,q_{+_{\gg}})} \left(\dR \hom_\pi \left(\mathscr{E}_{2}, \coker(\Phi)\otimes \omega_{\pi}[1]\right),  \LL_{\pi'}\right),
\end{align*}
where the second isomorphism is induced by the definition of $\pi^!$ and the third isomorphism is induced by Grothendieck-Verdier duality. So under the identification above, the reduced Atiyah class defines a morphism in the derived category
$$\alpha: \dR \hom_\pi \left(\mathscr{E}_{2}, \coker(\Phi)\otimes \omega_{\pi}[1]\right)\to  \LL_{\pi'}.$$ But by Grothendieck-Verdier duality again,
$$\dR \hom_\pi \left(\mathscr{E}_{2}, \coker(\Phi)\otimes \omega_{\pi}[1]\right)\cong \dR \hom_\pi \left(\coker(\Phi),\mathscr{E}_{2}\right)^\vee[-1] \cong\Fbr,$$ and hence we are done.


\textbf{Step 3:} \emph{(obstruction theory)} We show $h^0(\alpha)$ and $h^{-1}(\alpha)$ are respectively isomorphism and epimorphism.
Suppose we are in the situation of the diagram \eqref{sqzero}. Define $$f:=(\id,g):S\times T\to S\times \mathscr{M}^{(\overrightarrow{v}^{*}_{1},\overrightarrow{v}^{*}_{2})}(S,q_{+_{\gg}}).$$ 
Composing $e(\Tb)$ (given in \eqref{cotsqzero}) and the natural morphism of cotangent complexes $\dL g^*\LL_{\pi'} \to \LL_a$ gives the element $\varpi(g) \in \Ext^1_T(\dL g^*\LL_{\pi'},J)$ whose image under $\alpha$ is denoted by $$\alpha^*\varpi(g)\in \Ext^{1}_T\left(\dL g^*\Fbr, J\right).$$

For $i=0, 1$, we will use the following identifications:
 \begin{align*}
 \Ext^{i}_T\left(\dL g^*\Fbr, J\right)&\cong \Ext^{i}_T\left(\dL g^*\dR \hom_\pi\left(\mathscr{E}_{2}, \coker(\Phi)\otimes \omega_{\pi}[1]\right), J\right)\\
 &
\cong \Ext^{i}_{\mathscr{M}^{(\overrightarrow{v}^{*}_{1},\overrightarrow{v}^{*}_{2})}(S,q_{+_{\gg}})}\left( \dR \hom_\pi\left(\mathscr{E}_{2},  \coker(\Phi)\otimes \omega_{\pi}[1]\right), \dR g_*J\right)\\
&
\cong \Ext^{i}_{S\times \mathscr{M}^{(\overrightarrow{v}^{*}_{1},\overrightarrow{v}^{*}_{2})}(S,q_{+_{\gg}})}\left(\dR \hom \left(\mathscr{E}_{2}, \coker(\Phi)\otimes \omega_{\pi}[1]\right), \pi^!\dR g_*J\right)\\
&
\cong \Ext^{i+1}_{S\times \mathscr{M}^{(\overrightarrow{v}^{*}_{1},\overrightarrow{v}^{*}_{2})}(S,q_{+_{\gg}})}\left(\dR \hom \left(\mathscr{E}_{2}, \coker(\Phi)\right), \pi^* \dR g_*J\right).
\end{align*}
Here we have used the fact that $\dL g^* \dashv \dR g_*$ i.e. $\dL g^*$ is the left adjoint of $\dR g_*$, and Grothendieck-Verdier duality. Now  using $ \dR f_{*}\pi_T^*=\pi^* \dR g_*$  in the last $\Ext$ above, we get  

\begin{align}\label{isoms}
& \Ext^{i+1}_{S\times \mathscr{M}^{(\overrightarrow{v}^{*}_{1},\overrightarrow{v}^{*}_{2})}(S,q_{+_{\gg}})}\left(\dR \hom \left(\mathscr{E}_{2}, \coker(\Phi)\right), \dR f_{*}\pi_T^*J\right) \cong \quad \text{(by $\dL f^* \dashv \dR f_*$)\notag}\\
& \Ext^{i+1}_{S\times T}\left(\dL f^* \dR \hom\left(\mathscr{E}_{2}, \coker(\Phi)\right), \pi_T^*J\right)\cong \notag\\
& \Ext^{i+1}_{S\times T}\left( \dR\hom\left(\dL f^*\mathscr{E}_{2}, \dL f^*\coker(\Phi)\right), \pi_T^*J\right) \cong\notag\\
& \Ext^{i+1}_{S\times T}\left(\dL f^*\coker(\Phi), \pi_T^*J \otimes \dL f^*\mathscr{E}_{2}\right)\cong \quad (\text{by flatness of $\coker(\Phi)$ and $\mathscr{E}_{2}$})\notag\\
& \Ext^{i+1}_{S\times T}\left( \coker(\phi_T), \pi_T^*J \otimes  \mathscr{E}_{2,T}\right).
 \end{align}
 Similar to the Step 2 it can be seen that the composition $$\dL g^*\Fbr \xrightarrow{g^*\alpha}\dL g^*\LL_{\pi'} \to \LL_a$$ has a lift to the  reduced Atiyah class $\at_{\text{red}}(\phi_T)$ over $S\times T$ (see \eqref{atiyah}).
Therefore, $$\ob:=\at_{\text{red}}(\phi_T)\cup (\pi_{T}^{*}e(\Tb)\otimes \id),$$is isomorphic to the element $\alpha^{*}\varpi(g)$ via  isomorphisms \eqref{isoms} for  $i=1$. Now this implies that the $\alpha^{*}\varpi(g)$ vanishes if and only if there exists an extension $\bar{g}$ of $g$. Now if we follow isomorphisms in \eqref{isoms} for $i=0,$ we see that the set of such extensions $\bar{g}$ of $g$ forms a torsor under $\Hom(\dL g^*\Fbr, J)$. Now by \cite[Theorem 4.5]{BF97} $\alpha$ is an obstruction theory.

\end{proof}


\subsection{$2$-term absolute obstruction theory}  \label{2-step} 
As it was discussed, the choice for the Chern character vector $\overrightarrow{v}^{*}_{1}$ is special here and it satisfies the conditions of Assumption \ref{smoothness}. The first named author however is developing a modification of the construction below, which allows for  $\mathscr{M}^{\overrightarrow{v}^{*}_{1}}$ to be non-smooth, as long as it is equipped with a perfect deformation-obstruction theory of amplitude $[-1,0]$. This generalization will be discussed in future in a separate sequel to the current article. 
\begin{thm} \label{abs}
Consider the Assumption \ref{smoothness}. Then the moduli space $\mathscr{M}^{(\overrightarrow{v}^{*}_{1},\overrightarrow{v}^{*}_{2})}(S,q_{+_{\gg}})$ is equipped with a perfect absolute obstruction theory $$\Fb\to  \LL_{\mathscr{M}^{(\overrightarrow{v}^{*}_{1},\overrightarrow{v}^{*}_{2})}(S,q_{+_{\gg}})} $$ induced by taking an (absolute) cone over the perfect relative obstruction theory of Theorem \ref{rel}. The dual of $\Fb$ is represented in the derived category by 
\begin{align}\label{absolute-cone}
\Fbv\cong\cone \bigg( \big[\dR\hom_\pi \left(\mathscr{E}_{1},\mathscr{E}_{1}\right) &  \oplus \dR\hom_\pi\left(\mathscr{E}_{2},\mathscr{E}_{2}\right)\big]_0 \xrightarrow{[(\Xi',\Xi)]} \dR \hom_\pi \left(\mathscr{E}_{1}, \mathscr{E}_{2}\right)\bigg),\end{align}
where $[-]_0$ means the trace-free part.
\end{thm}
\begin{proof}  The procedure we will elaborate below is well known to the experts in the field, and all ideas of the cone construction we are about to discuss, are copied from insightful work of Maulik, Pandharipande and Thomas \cite[Section 3.5]{MPT10}. The less familiar part of the proof however is to show that the absolute deformation-obstruction complex takes the form on the right hand side of \eqref{absolute-cone}. Consider the morphism $\pi'$ in Diagram \eqref{sqzero}. The moduli space $\mathscr{M}^{\overrightarrow{v}^{*}_{1}}$ is smooth by assumption, and hence $$E^{\bullet}_{1}=\mathbb{L}^{\bullet}_{\mathscr{M}^{\overrightarrow{v}^{*}_{1}}}\cong \Omega_{\mathscr{M}^{\overrightarrow{v}^{*}_{1}}}.$$


Then it is immediately implied that there exists a diagram 
\begin{equation} \label{lifting}
\xymatrix@C=35pt@R=15pt{
\pi'^{*}E^{\bullet}_{1} \ar@{.>}[r] \ar^{=}[d]& F^{\bullet}:=\text{Cone}(\theta)[-1] \ar@{.>}[r]\ar@{.>}^{\epsilon}[d]&F^{\bullet}_{rel}\ar^{\alpha}[d]\ar^-{\theta:=c\circ \alpha}[r]&\pi'^{*}E^{\bullet}_{1}[1]\ar^{=}[d]\\\ \pi'^{*}\Omega_{\mathscr{M}^{\overrightarrow{v}^{*}_{1}}}\ar[r]& \LL_{\mathscr{M}^{(\overrightarrow{v}^{*}_{1},\overrightarrow{v}^{*}_{2})}(S,q_{+_{\gg}})} \ar[r]&\LL_{\pi'}\ar^{c}[r]&\pi'^{*}\Omega_{\mathscr{M}^{\overrightarrow{v}^{*}_{1}}}[1].}
\end{equation}

Now taking the long exact sequence of cohomologies of the upper and lower horizontal rows of \eqref{lifting} we obtain

\begin{align} \label{lifting2}
&
\xymatrix@C=25pt@R=15pt{
0 \ar[r] \ar[d]& h^{-1}(F^{\bullet})\ar[r]\ar^{h^{-1}(\epsilon)}[d]&h^{-1}(F^{\bullet}_{rel})\ar^{h^{-1}(\alpha)}@{->>}[d]\ar[r]&h^{-1}(\pi'^{*}E^{\bullet}_{1}[1]) \ar^{\cong}[d]\ar[r]&.\\ 0\ar[r]& h^{-1}(\LL_{\mathscr{M}^{(\overrightarrow{v}^{*}_{1},\overrightarrow{v}^{*}_{2})}(S,q_{+_{\gg}})})\ar[r]&h^{-1}(\LL_{\pi'})\ar[r]&h^{-1}(\pi'^{*}\Omega_{\mathscr{M}^{\overrightarrow{v}^{*}_{1}}}[1])\ar[r]&.}\notag\\
&
\xymatrix@C=15pt@R=15pt{
.\ar[r]&h^{0}(F^{\bullet})\ar[r] \ar^{h^{0}(\epsilon)}[d]& h^{0}(F^{\bullet}_{rel}) \ar[r]\ar^{\cong}_{h^{0}(\alpha)}[d]&0\ar[d]\\\ .\ar[r]&h^{0}(\LL_{\mathscr{M}^{(\overrightarrow{v}^{*}_{1},\overrightarrow{v}^{*}_{2})}(S,q_{+_{\gg}})})\ar[r]& h^{0}(\LL_{\pi'})\ar[r]&0}\notag\\
\end{align}
Now by using the 4 lemma twice, once for the first 4 columns of \eqref{lifting2} on the left (the column of zeros on the left is considered as the first column) and the 4-lemma for the columns 3,4,5,6 it can be seen that $h^{-1}(\phi)$ is an epimorphism and $h^{0}(\phi)$ is an isomorphism. Hence, $F^{\bullet}$ satisfies the cohomological conditions of being a deformation obstruction theory. The perfectness of $F^{\bullet}$ can also be deduced  by considering the long exact sequence of cohomologies in $h^{-2}$ and $h^{1}$ levels. Now denote by $\mathscr{R}$ the right hand side of the expression in \eqref{absolute-cone}. We would like to show that $\Fbv$ takes the form $\mathscr{R}$. By Theorem \ref{rel}, $$\Fbrv=\cone\left(\dR\hom_\pi\left(\mathscr{E}_{2},\mathscr{E}_{2}\right)\to \dR \hom_\pi \left(\mathscr{E}_{1}, \mathscr{E}_{2}\right)\right),$$ so we need to show that the map 

\begin{equation}\label{the-map}
\Fbr\xrightarrow{\theta} \pi'^{*}\Omega_{\mathscr{M}^{\overrightarrow{v}^{*}_{1}}}[1]
\end{equation}is induced by a composite morphism

\begin{align}\label{claim}
&\Fbr\cong \dR \hom_\pi\left(\mathscr{E}_{2},  \coker(\Phi)\otimes \omega_{\pi}\right)[1]\to \dR \hom_\pi\left(\mathscr{E}_{1},  \coker(\Phi)\otimes \omega_{\pi}\right)[1]\notag\\
&
\to \dR \hom_\pi\left(\mathscr{E}_{1},  \mathscr{E}_{1}\otimes \omega_{\pi}\right)_{0}[2]\cong \pi'^{*}\Omega_{\mathscr{M}^{\overrightarrow{v}^{*}_{1}}}[1],
\end{align}
where the first and second morphisms are induced by applying $$\dR \hom_\pi\left(-,  \coker(\Phi)\otimes \omega_{\pi}\right)[1]$$ to the morphism $\Phi:  \mathscr{E}_{1} \to  \mathscr{E}_{2}$, and  $\coker(\Phi)\to  \mathscr{E}_{1}[1]$ respectively. To prove this, consider the diagram of graded sheaves of algebras on $S\times \mathscr{M}^{(\overrightarrow{v}^{*}_{1},\overrightarrow{v}^{*}_{2})}(S,q_{+_{\gg}})$

\begin{equation}\label{transitivity}
 \xymatrix@=2em{\O_{S\times \mathscr{M}^{\overrightarrow{v}^{*}_{1}}} \ar[r]& \O_{S\times \mathscr{M}^{(\overrightarrow{v}^{*}_{1},\overrightarrow{v}^{*}_{2})}_{\text{T}}(S,q)}\oplus \mathscr{E}_{1} \ar^{(\id, \Phi)}[r]& \O_{S\times \mathscr{M}^{(\overrightarrow{v}^{*}_{1},\overrightarrow{v}^{*}_{2}, )}_{\text{T}}(S,q)}\oplus \mathscr{E}_{2}\\
 \O_{S } \ar[r] \ar[u]& \O_{S\times \mathscr{M}^{(\overrightarrow{v}^{*}_{1},\overrightarrow{v}^{*}_{2})}(S,q_{+_{\gg}})}  \ar[r] \ar[u]& \O_{S\times \mathscr{M}^{(\overrightarrow{v}^{*}_{1},\overrightarrow{v}^{*}_{2})}(S,q_{+_{\gg}})}\oplus \mathscr{E}_{2},\ar@{=}[u]\\
 \O_{S} \ar[r] \ar@{=}[u]& \O_{S\times \mathscr{M}^{(\overrightarrow{v}^{*}_{1},\overrightarrow{v}^{*}_{2})}(S,q_{+_{\gg}})}  \ar[r] \ar@{=}[u]& \O_{S\times \mathscr{M}^{(\overrightarrow{v}^{*}_{1},\overrightarrow{v}^{*}_{2})}(S,q_{+_{\gg}})}\oplus \mathscr{E}_{1},\ar[u]&
 } 
\end{equation}

Here, for each vertex on the diagram with two summands, the first summand is in degree zero, and the second summand is in degree 1\footnote{Hence following the insightful idea of Illusie \cite[IV.2.3]{Ill} we lift the problem to an equivariant setting, where the sheaves and their morphisms will be realized as graded modules over a $\C^*$-graded algebra and their corresponding graded morphisms.}. Associated to top two rows of the diagram above is the \textit{exact triangle of transitivity}\footnote{This terminology is introduced by Illusie.} of the relative \textbf{graded} cotangent complexes (see \cite[IV.2.3]{Ill}), which leads to definition of $\at_{\text{red}}(\Phi)$

\begin{equation}\label{square1}
\xymatrix@C=1em{
\coker(\Phi)  \ar^{\at_{\text{red}}(\Phi)}[r] & \pi^*\LL_{\pi'}\otimes \mathscr{E}_{2}[1] \\
\mathscr{E}_{2} \ar[u]  \ar[r] & \LL_{\mathscr{M}^{(\overrightarrow{v}^{*}_{1},\overrightarrow{v}^{*}_{2})}(S,q_{+_{\gg}})} \otimes \mathscr{E}_{2}[1]. \ar[u]}
\end{equation}
Here, for simplicity, we have neglected the notation for pullbacks via natural morphisms. The bottom row of diagram above is induced by the composite morphism $$P_{\cE_{2}}^*\mathscr{E}_{2} \to P_{\cE_{2}}^*\LL_{\mathscr{M}^{\overrightarrow{v}^{*}_{2}}}[1]\to \LL_{\mathscr{M}^{(\overrightarrow{v}^{*}_{1},\overrightarrow{v}^{*}_{2})}(S,q_{+_{\gg}})}[1]$$where the first morphism is induced by the pullback of the Atiyah class of $\mathscr{E}_{2}$ over $\mathscr{M}^{\overrightarrow{v}^{*}_{2}}$ and the second morphism is induced by the natural morphism $P_{\cE_{2}}:=\mathscr{M}^{(\overrightarrow{v}^{*}_{1},\overrightarrow{v}^{*}_{2})}(S,q_{+_{\gg}})\to \mathscr{M}^{\overrightarrow{v}^{*}_{2}}$. Now taking the cone of the vertical arrows in Diagram \eqref{square1} we obtain a commutative square diagram

\begin{equation}\label{square2}
\xymatrix@C=3em{
\mathscr{E}_{1}[1]  \ar[r] & \pi'^{*}\Omega_{\mathscr{M}^{\overrightarrow{v}^{*}_{1}}}\otimes \mathscr{E}_{2}[2] \\
\coker(\Phi)\ar[u]  \ar^{\at_{\text{red}}(\Phi)}[r] & \LL_{\pi'} \otimes \mathscr{E}_{2}[1]. \ar[u]}
\end{equation}
where again, using the transitivity triangles of the bottom two rows of  Diagram \eqref{transitivity}, it can be seen that the top row of Diagram \eqref{square2} is identified with the composite morphism 
\begin{equation}\label{theta-compos}
 \mathscr{E}_{1}\xrightarrow{\at(\mathscr{E}_{1})}\pi'^{*}\Omega_{\mathscr{M}^{\overrightarrow{v}^{*}_{1}}}\otimes \mathscr{E}_{1}[1]\xrightarrow{\id\otimes \Phi}\pi'^{*}\Omega_{\mathscr{M}^{\overrightarrow{v}^{*}_{1}}}\otimes \mathscr{E}_{2}[1].
\end{equation}
The claim about map $\theta$ given by \eqref{claim} now follows from \eqref{theta-compos} and the construction of the map $\alpha$ using the reduced Atiyah class $\at_{\text{red}}(\Phi)$ in Step 2 of the Theorem \ref{rel}. Now using \eqref{claim} and the cone construction in \eqref{lifting} we can conclude  that $F^{\bullet\vee}$ can be rewritten as 
\begin{align*}
\Fbv=&\cone \left(\dR\hom_\pi \left(\mathscr{E}_{1},\mathscr{E}_{1}\right)_0 \xrightarrow{\theta^\vee} \cone\bigg(\dR\hom_\pi\left(\mathscr{E}_{2},\mathscr{E}_{2}\right)\to \dR \hom_\pi \left(\mathscr{E}_{1}, \mathscr{E}_{2}\right)\bigg)\right).
\end{align*} 

%
%

Consider the commutative diagram 
$$\xymatrix@=1em@R=15pt{
\dR\hom_\pi\left(\mathscr{E}_{2},\mathscr{E}_{2}\right)   \ar[r] &  \dR\hom_\pi \left(\mathscr{E}_{1},\mathscr{E}_{1}\right)\oplus \dR\hom_\pi\left(\mathscr{E}_{2},\mathscr{E}_{2}\right) \ar[r] & \dR\hom_\pi \left(\mathscr{E}_{1},\mathscr{E}_{1}\right)\\
& \dR\pi_* \O_{S \times \mathscr{M}^{(\overrightarrow{v}^{*}_{1},\overrightarrow{v}^{*}_{2})}(S,q_{+_{\gg}})}  \ar[u]^{[\id \,\,\id ]^t}  \ar@{=}[r] & \dR\pi_* \O_{S \times \mathscr{M}^{(\overrightarrow{v}^{*}_{1},\overrightarrow{v}^{*}_{2})}(S,q_{+_{\gg}})} \ar[u]^{\id}.}$$
In which both rows are exact triangles. Now dividing the first row by the second row we obtain an exact triangle
\begin{equation}\label{trace-less}
\dR\hom_\pi\left(\mathscr{E}_{2},\mathscr{E}_{2}\right) \to [\dR\hom_\pi \left(\mathscr{E}_{1},\mathscr{E}_{1}\right)\oplus \dR\hom_\pi\left(\mathscr{E}_{2},\mathscr{E}_{2}\right)]_0\to \dR\hom_\pi \left(\mathscr{E}_{1},\mathscr{E}_{1}\right)_0
\end{equation}

On the other hand, there  exists a commutative diagram 
\begin{equation}\label{no-trace}
\mbox{\scriptsize\xymatrix@C=1em{
\dR\hom_\pi\left(\mathscr{E}_{2},\mathscr{E}_{2}\right)  \ar@{=}[d] \ar[r] &  [\dR\hom_\pi \left(\mathscr{E}_{1},\mathscr{E}_{1}\right)\oplus \dR\hom_\pi\left(\mathscr{E}_{2},\mathscr{E}_{2}\right)] \ar[d]^{[(-\Xi',\Xi)]} \ar[r] & \dR\hom_\pi \left(\mathscr{E}_{1},\mathscr{E}_{1}\right)\oplus \dR\pi_* \O_{S \times \mathscr{M}^{(\overrightarrow{v}^{*}_{1},\overrightarrow{v}^{*}_{2})}(S,q_{+_{\gg}})}\ar[d]^{[\theta^\vee\,\, 0]}\\
\dR\hom_\pi\left(\mathscr{E}_{2},\mathscr{E}_{2}\right) \ar[r]^{\Xi} & \dR \hom_\pi \left(\mathscr{E}_{1}, \mathscr{E}_{2}\right)  \ar[r] &\cone \left(\dR\hom_\pi\left(\mathscr{E}_{2},\mathscr{E}_{2}\right)\to \dR \hom_\pi \left(\mathscr{E}_{1}, \mathscr{E}_{2}\right)\right)}}
\end{equation}

in which the two rows are the natural exact triangles. Notice that $$\dR\hom_\pi \left(\mathscr{E}_{1},\mathscr{E}_{1}\right)\cong \dR\hom_\pi \left(\mathscr{E}_{1},\mathscr{E}_{1}\right)_{0}\oplus \dR\pi_{*}\O$$ and hence it can be seen that the vertical arrows in Diagram \eqref{no-trace} factor through exact triangle \eqref{trace-less}. Hence, taking the cone of the diagram one gets the exact triangle that fits into the following commutative diagram in which all the rows and columns are exact triangles:
$$\mbox{\scriptsize\xymatrix@C=1em{
\dR\hom_\pi\left(\mathscr{E}_{2},\mathscr{E}_{2}\right)  \ar@{=}[d] \ar[r] &  [\dR\hom_\pi \left(\mathscr{E}_{1},\mathscr{E}_{1}\right)\oplus \dR\hom_\pi\left(\mathscr{E}_{2},\mathscr{E}_{2}\right)]_0 \ar[d]^{[(-\Xi',\Xi)]} \ar[r] & \dR\hom_\pi \left(\mathscr{E}_{1},\mathscr{E}_{1}\right)_0 \ar[d]^{[\theta^\vee\,\, 0]}\\
\dR\hom_\pi\left(\mathscr{E}_{2},\mathscr{E}_{2}\right) \ar[r]^{\Xi} & \dR \hom_\pi \left(\mathscr{E}_{1}, \mathscr{E}_{2}\right) \ar[d] \ar[r] &\cone \left(\dR\hom_\pi\left(\mathscr{E}_{2},\mathscr{E}_{2}\right)\to \dR \hom_\pi \left(\mathscr{E}_{1}, \mathscr{E}_{2}\right)\right) \ar[d]  \\
& \mathscr{R} \ar[r]  & \Fbv.}}
$$
Here, the columns and the top and middle rows are exact triangles with commutative top squares, and the bottom row is induced from the rest of the diagram by taking the cone. Therefore, the bottom row must also be an exact triangle which means that $\mathscr{R}\cong \Fbv$ as claimed.
This finishes the proof of Theorem \ref{abs}.

\end{proof}
Theorems \ref{rel} and \ref{abs} imply 
\begin{cor} \label{virclass}
The perfect obstruction theory $\Fb$ defines a virtual fundamental class over $\mathscr{M}^{(\overrightarrow{v}^{*}_{1},\overrightarrow{v}^{*}_{2})}(S,q_{+_{\gg}})$ denoted by $$[\mathscr{M}^{(\overrightarrow{v}^{*}_{1},\overrightarrow{v}^{*}_{2})}(S,q_{+_{\gg}})]^{\vir}\in A_d(\mathscr{M}^{(\overrightarrow{v}^{*}_{1},\overrightarrow{v}^{*}_{2})}(S,q_{+_{\gg}})), \quad  \quad d=\text{rk}(\Fb).$$
\end{cor}
\begin{rmk}
In cases where $p_{g}(S)>0$, it might be true that the virtual cycle in Corollary  \ref{virclass} vanishes due to existence of a surjective co-section for the obstruction bundle induced by the absolute deformation-obstruction theory $\Fb$. In this case one can define the reduced obstruction theory using co-section localization techniques \cite{KL13}. We will discuss an instant of this construction later in Section \ref{sec:reduced}, and for now generalize our construction to higher length flags. 
\end{rmk}

\section{Higher length flags}\label{high-length}
Now let us consider the more general case of torsion free higher length flags $$E_{1}\to E_{2}\to\cdots\to E_{r},$$where the first term in the flag satisfies the conditions of Assumption \ref{smoothness}. First we construct our moduli space out of gluing holomorphic triples $E_{i}\xrightarrow{\phi_{i}} E_{i+1}$ which are $q_{+_{\gg}}$-stable.

Define $\mathscr{M}^{(\overrightarrow{v}^{*}_{i},\overrightarrow{v}^{*}_{i+1})}_{\text{T}}(S,q_{+_{\gg}})$ to be the projective moduli schemes of $q_{+_{\gg}}$-stable triples, constructed in Section \ref{chap2} with fixed numerical data $(\overrightarrow{v}^{*}_{i},\overrightarrow{v}^{*}_{i+1}), i=1, \cdots, r-1$, such that $\overrightarrow{v}^{*}_{1}$ satisfies the condition of Assumption \ref{smoothness}. Fix an integer  $i=2, \cdots , r-1$, there exists a natural projection map $\pi_{i}$, defined as follows $$\pi_{i}:=\mathscr{M}^{(\overrightarrow{v}^{*}_{i-1},\overrightarrow{v}^{*}_{i})}_{\text{T}}(S,q)\times \mathscr{M}^{(\overrightarrow{v}^{*}_{i},\overrightarrow{v}^{*}_{i+1})}_{\text{T}}(S,q)\to \mathscr{M}^{\overrightarrow{v}^{*}_{i}}\times \mathscr{M}^{\overrightarrow{v}^{*}_{i}},$$sending $\Bigg( [E_{i-1} \xrightarrow{\phi_{i-1}}E_{i}], [E'_{i} \xrightarrow{\phi_{i}}E'_{i+1}]\Bigg)\to (E_{i}, E'_{i})$. Now consider the diagonal embedding morphisms $$\Delta_{i}: \mathscr{M}^{\overrightarrow{v}^{*}_{i}}\to \mathscr{M}^{\overrightarrow{v}^{*}_{i}}\times \mathscr{M}^{\overrightarrow{v}^{*}_{i}}, i=2, \cdots, r-1.$$The fibered product $$\mathscr{M}^{(\overrightarrow{v}^{*}_{i-1},\overrightarrow{v}^{*}_{i})}_{\text{T}}(S,q)\times_{\Delta_{i}} \mathscr{M}^{(\overrightarrow{v}^{*}_{i},\overrightarrow{v}^{*}_{i+1})}_{\text{T}}(S,q)$$parametrizes tuples $$E_{i-1}\hookrightarrow E_{i}\hookrightarrow E_{i+1}$$where the injective morphisms are induced by $q_{+_{\gg}}$-stability condition on $\mathscr{M}^{(\overrightarrow{v}^{*}_{i-1},\overrightarrow{v}^{*}_{i})}_{\text{T}}(S,q_{+_{\gg}})$ and $\mathscr{M}^{(\overrightarrow{v}^{*}_{i},\overrightarrow{v}^{*}_{i+1})}_{\text{T}}(S,q_{+_{\gg}})$ respectively. Now we can re-iterate this process and obtain the moduli scheme parametrizing a nested flag of $q_{+_{\gg}}$-stable torsion-free triples of length $r$: 
 \begin{defi}(\textit{Moduli scheme of holomorphic $r$-tuples})\label{T-def2}
 Given $\overrightarrow{\textbf{v}}^{*}:=(\overrightarrow{v}^{*}_{1}, \cdots \overrightarrow{v}^{*}_{r})$ define the moduli scheme, $\mathscr{M}^{(\overrightarrow{v}^{*}_{1},\cdots, \overrightarrow{v}^{*}_{r})}_{\text{T}}(S,q_{+_{\gg}})$ of holomorphic $r-tuples$ parameterizing $$E_{1}\hookrightarrow E_{2}\hookrightarrow \cdots\hookrightarrow E_{r}$$ with $\text{Ch}(E_{i}):=\overrightarrow{v}^{*}_{i}$ as the $r-1$-fold fibered product 
 
 \begin{align*}
 &\mathscr{M}^{(\overrightarrow{v}^{*}_{1},\cdots, \overrightarrow{v}^{*}_{r})}_{\text{T}}(S,q_{+_{\gg}}):=\notag\\
 &
 \mathscr{M}^{(\overrightarrow{v}^{*}_{1},\overrightarrow{v}^{*}_{2})}(S,q_{+_{\gg}})\times _{\Delta_{1}}\mathscr{M}^{(\overrightarrow{v}^{*}_{2},\overrightarrow{v}^{*}_{3})}_{\text{T}}(S,q_{+_{\gg}})\times\cdots\times_{\Delta_{r-1}} \mathscr{M}^{(\overrightarrow{v}^{*}_{r-1},\overrightarrow{v}^{*}_{r})}_{\text{T}}(S,q_{+_{\gg}})
 \end{align*}
 given by the diagram
 \begin{equation}  \label{products}
\xymatrix{
\mathscr{M}^{(\overrightarrow{v}^{*}_{1},\cdots, \overrightarrow{v}^{*}_{r})}_{\text{T}}(S,q_{+_{\gg}}) \ar[d] \ar[r] & \prod_{i=2}^{r-1}\mathscr{M}^{(\overrightarrow{v}^{*}_{i},\overrightarrow{v}^{*}_{i+1})}_{\text{T}}(S,q_{+_{\gg}}) \ar[d]^{\prod_{i=2}^{r-1}\pi_{i}}\\
\prod_{i=2}^{r-1}\Delta_{i}:=\prod_{i=2}^{r-1}\mathscr{M}^{\overrightarrow{v}^{*}_{i}} \ar@{^{(}->}[r]& \prod_{i=2}^{r-1}\mathscr{M}^{\overrightarrow{v}^{*}_{i}}\times \mathscr{M}^{\overrightarrow{v}^{*}_{i}}} 
\end{equation}
 \end{defi}
\begin{rmk}
By Lemma \ref{rk1} the $q_{+_{\gg}}$-stability of $E_{i}\xrightarrow{\phi} E_{i+1}$ implies that each $\phi_{i}$ is given by an injective morphism for all $ i=1,\cdots, r-1$ and moreover, by our construction and \cite[Remark 1.4]{S00}, the $q_{+_{\gg}}$-stability of tuples $(E_{i}, E_{i+1}, \phi_{i})$ for all $i$, implies the $q_{+_{\gg}}$-stability of the $r$-tuple $E_{1}\to\cdots\to E_{r}$. The latter, following the construction of \cite{S00}, implies that $\mathscr{M}^{(\overrightarrow{v}^{*}_{1},\cdots, \overrightarrow{v}^{*}_{r})}_{\text{T}}(S,q)$ is given as a projective scheme.
\end{rmk}

We will now discuss the iterative procedure for constructing the deformation obstruction theory for higher length flags in Definition \ref{T-def2}. Now for simplicity and without loss of generality assume that in diagram \eqref{products} we have that $r=3$. Consider the chain of natural forgetful morphisms and the associated exact triangle of cotangent complexes
\begin{equation} \label{forget}\mathscr{M}^{(\overrightarrow{v}^{*}_{1},\overrightarrow{v}^{*}_{2}, \overrightarrow{v}^{*}_{3})}_{\text{T}}(S,q_{+_{\gg}}) \xrightarrow{f_2} \mathscr{M}^{(\overrightarrow{v}^{*}_{1},\overrightarrow{v}^{*}_{2})}(S,q_{+_{\gg}})\xrightarrow{f_1} \mathscr{M}^{\overrightarrow{v}^{*}_{1}},\quad  \LL_{f_2}[-1]\xrightarrow{j_2} \dL f_2^* (\LL_{f_1})\xrightarrow{j_1} \LL_f \xrightarrow{j_3} \LL_{f_2}\end{equation} where $f:=f_1\circ f_2$.

Theorem \ref{rel}, provides the relative perfect obstruction theory for the morphism $f_1$, that we denote by \begin{equation} \label{rel2-1}\Fb_{f_1}\xrightarrow{\alpha_1} \LL_{f_1}.\end{equation}

\begin{lem} \label{rel3-2}
There exists a relative perfect obstruction theory $\Fb_{f_2}\xrightarrow{\alpha_2} \LL_{f_2}$, where
\begin{equation}\label{Ff2}\Fb_{f_2}=\dR\hom_\pi \left(\mathscr{E}_{3}, \coker(\Phi_2)\otimes \omega_S\right )[1]  \end{equation}

\end{lem}

\begin{proof} The proof is similar to the proof of Theorem \ref{rel} (see Step 2 of that proof for the corresponding expression in RHS of \eqref{Ff2}). This time the obstruction theory is associated to the universal map $\mathscr{E}_{1}\xrightarrow{\Phi_{1}}\mathscr{E}_{2}\xrightarrow{\Phi_{2}} \mathscr{E}_{3}$ while the data $(\mathscr{E}_{1}, \mathscr{E}_{2}, \Phi_1)$ is kept fixed. 
\end{proof}

\begin{lem} \label{commut}
The obstruction theories $\Fb_{f_1}$ and $\Fb_{f_2}$ fit into the following commutative diagram:

\begin{equation}\label{diagcomm}
\xymatrix{
\Fb_{f_2}[-1] \ar^{\alpha_2[-1]}[d]\ar[r]^r &\dL f_2^*( \Fb_{f_1}) \ar^{f_2^*( \alpha_1)}[d] \\  \LL_{f_2}[-1] \ar[r]^{j_2} &\dL f_2^* (\LL_{f_1}).} 
\end{equation}
\end{lem}

\begin{proof}
\textbf{Step 1:} \emph{(Defining the map $r$)} The maps in diagram \eqref{diagcomm} except $r$ are already defined above (see \eqref{forget}, \eqref{rel2-1}, and Lemma \ref{rel3-2}).  Now by the universal properties of the the moduli spaces of holomorphic $r$-tuples, we can write

\begin{equation}\label{Ff1}\dL f_2^*( \Fb_{f_1})= \dR\hom_\pi \left(\mathscr{E}_{2}, \coker(\Phi_1)\otimes \omega_S[1]\right )\end{equation}


The chain of maps $\mathscr{E}_{1}\xrightarrow{\Phi_1} \mathscr{E}_{2}\xrightarrow{\Phi_2} \mathscr{E}_{3}$ induces an exact triangle 
\begin{equation}\label{cokers}\coker(\Phi_2\circ \Phi_1)\xrightarrow{i_3} \coker(\Phi_2)\xrightarrow{i_2[1]} \coker( \Phi_1) [1].\end{equation}
Applying $\dR\hom_\pi \left(-, \coker(\Phi_2)\otimes \omega_S\right )$ to $\Phi_{2}: \mathscr{E}_{2}\to \mathscr{E}_{3}$, and composing with $i_2[1]$ we obtain 
\begin{align} \label{define-r}
&\Fb_{f_2}[-1]=\dR\hom_\pi \left(\mathscr{E}_{3}, \coker(\Phi_2)\otimes \omega_S\right )\to \dR\hom_\pi \left(\mathscr{E}_{2}, \coker(\Phi_2)\otimes \omega_S\right ) \to \notag \\
&\dR\hom_\pi \left(\mathscr{E}_{2}, \coker(\Phi_1)\otimes \omega_S[1]\right )=\dL f_2^*(\Fb_{f_1}). \end{align}
 The map $r$ in diagram \eqref{diagcomm} is then defined by composition of two maps in (\ref{define-r}). This part of the proof is exactly the same as \cite[Lemma 2.12]{GSY17a}. There, the proof works since one is working with nested ideal sheaves, and here we obtain the desired outcome, since stability forces the morphisms $\Phi_{i}$ in each level to be injective. We will reproduce (an adaptation to our case of the) proof used in \cite[Lemma 2.12]{GSY17a}.\\\\
\textbf{Step 2:} (\emph{Commutativity of diagram (\ref{diagcomm}))}  We start with the following diagram in which the columns 
are the exact triangles \eqref{cokers} and \eqref{forget}: 
\begin{equation} \label{atiyah-commut}
\xymatrix@C=85pt@R=25pt{
\coker( \Phi_1) [1]  \ar^-{(\id\otimes \Phi_2)\circ \at_{\text{red}}(\Phi_1 )[1]}[r] &\pi^* \dL f_2^*\left(\LL_{f_1}\right)[1]\otimes \mathscr{E}_{3}[1] \\
\coker(\Phi_2)  \ar^-{\at(\Phi_2 )}[r] \ar[u]_{i_2[1]}&\pi^* \LL_{f_2}\otimes \mathscr{E}_{3}[1] \ar[u]_{(\pi^*j_2[1]\otimes \id)[1]}\\
\coker(\Phi_2\circ \Phi_1)  \ar^-{\at(\Phi_2\circ \Phi_1)}[r] \ar[u]_{i_3}& \pi^*\LL_{f}\otimes \mathscr{E}_{3}[1] \ar[u]_{(\pi^*j_3\otimes \id)[1]}\\
\coker( \Phi_1) \ar[u]_{i_1}  \ar^-{(\id\otimes \Phi_2)\circ\at(\Phi_1)}[r] & \pi^*\dL f_2^*\left(\LL_{f_1}\right)\otimes \mathscr{E}_{3}[1] \ar[u]_{(\pi^*j_1\otimes \id)[1]}} \end{equation}

We prove diagram \eqref{atiyah-commut} is commutative. For this, once again consider the following the commutative diagrams of sheaf of graded algebras:
 
 $$ \xymatrix@=1em{\O_{S\times \mathscr{M}^{\overrightarrow{v}^{*}_{1}}} \ar[r]& \O_{S\times \mathscr{M}^{(\overrightarrow{v}^{*}_{1},\overrightarrow{v}^{*}_{2}, \overrightarrow{v}^{*}_{3})}_{\text{T}}(S,q)}\oplus \mathscr{E}_{1} \ar[r]& \O_{S\times \mathscr{M}^{(\overrightarrow{v}^{*}_{1},\overrightarrow{v}^{*}_{2}, \overrightarrow{v}^{*}_{3})}_{\text{T}}(S,q)}\oplus \mathscr{E}_{3}\\
 \O_{S \times \mathscr{M}^{\overrightarrow{v}^{*}_{1}}} \ar[r] \ar@{=}[u]& \O_{S\times \mathscr{M}^{(\overrightarrow{v}^{*}_{1},\overrightarrow{v}^{*}_{2})}(S,q_{+_{\gg}})} \oplus \mathscr{E}_{1} \ar[r] \ar[u]& \O_{S\times \mathscr{M}^{(\overrightarrow{v}^{*}_{1},\overrightarrow{v}^{*}_{2})}(S,q_{+_{\gg}})}\oplus \mathscr{E}_{2},\ar[u]&
 } $$ and  $$
 \xymatrix@=1em{\O_{{ S\times \mathscr{M}^{(\overrightarrow{v}^{*}_{1},\overrightarrow{v}^{*}_{2})}(S,q_{+_{\gg}})}} \ar[r]& \O_{S\times \mathscr{M}^{(\overrightarrow{v}^{*}_{1},\overrightarrow{v}^{*}_{2}, \overrightarrow{v}^{*}_{3})}_{\text{T}}(S,q)}\oplus \mathscr{E}_{2} \ar[r]& \O_{S\times \mathscr{M}^{(\overrightarrow{v}^{*}_{1},\overrightarrow{v}^{*}_{2}, \overrightarrow{v}^{*}_{3})}_{\text{T}}(S,q)}\oplus \mathscr{E}_{3}\\
 \O_{S \times \mathscr{M}^{\overrightarrow{v}^{*}_{1}}} \ar[r] \ar[u]& \O_{S\times \mathscr{M}^{(\overrightarrow{v}^{*}_{1},\overrightarrow{v}^{*}_{2}, \overrightarrow{v}^{*}_{3})}_{\text{T}}(S,q)} \oplus \mathscr{E}_{1} \ar[r] \ar[u]& \O_{S\times \mathscr{M}^{(\overrightarrow{v}^{*}_{1},\overrightarrow{v}^{*}_{2}, \overrightarrow{v}^{*}_{3})}_{\text{T}}(S,q)}\oplus \mathscr{E}_{3},\ar@{=}[u]&
 }$$
Here, for each vertex on the middle and right columns of both diagrams, the first summand is in degree zero, and the second summand is in degree 1. Once again, we will take advantage of Illusie's exact triangles of transitivity of the relative graded cotangent complexes associated to each row of the diagram (see \cite[IV.2.3]{Ill}). The vertical maps in the two diagrams above, induce the natural maps between the corresponding graded cotangent complexes in the exact triangles of transitivity, in such a way that all the resulting squares are commutative. Let us again take advantage of the operator $k^i(-), i=0,1$, which takes the degree $i$ graded piece of a graded object. Applying $k^1(-)$ to the above two diagrams, we get the commutativity of the following two squares:

$$\xymatrix@C=1em{
\coker(\Phi_2\circ \Phi_1)  \ar[r] & \left(\pi^*\LL_{f}\otimes \mathscr{E}_{3}[1]\right) \oplus \mathscr{E}_{2}[1] \\
\coker( \Phi_1) \ar[u]_{i_1}  \ar[r] & \left(\pi^*\dL f_2^*\left(\LL_{f_1}\right)\otimes \mathscr{E}_{3}[1]\right) \oplus \mathscr{E}_{1}[1], \ar[u]_{(\pi^*j_1\otimes \id\oplus \Phi_1)[1]} }
$$ and

$$\xymatrix@C=1em{
\coker(\Phi_2)  \ar[r] &\left(\pi^* \LL_{f_2}\otimes \mathscr{E}_{3}[1]\right) \oplus \mathscr{E}_{3}[1]\\
\coker(\Phi_2\circ \Phi_1)  \ar[r] \ar[u]_{i_3}& \left(\pi^*\LL_{f}\otimes \mathscr{E}_{3}[1]\right) \oplus  \mathscr{E}_{2}[1] \ar[u]_{(\pi^*j_3\otimes \id\oplus \Phi_2)[1]}.}
$$
Here in both square diagrams, the horizontal arrows are parts of the triangles of transitivity (after taking the degree 1 graded pieces). 

\begin{rmk}
Here we have used the base change property for the cotangent complexes, as already employed in Section \ref{2-step}  (see \cite[II.2.2]{Ill}), as well as the following isomorphisms
$$k^0\left(\mathbb{L}^{\bullet,\text{gr}}_{(B_0\oplus B_1)/(A_0\oplus A_1)}\right)\cong \LL_{B_0/A_0}, \quad k^1\left(\mathbb{L}^{\bullet,\text{gr}}_{(A_0\oplus C_1)/(A_0\oplus A_1)}\right)\cong \coker(s),$$ where $A_0\oplus A_1\to B_0\oplus B_1$ and $A_0\oplus A_1\xrightarrow{\id \oplus s} A_0\oplus C_1$ are the homomorphism of graded $\C$-algebras (summands with index $i$ are in degree $i$), and furthermore $s$ is injective. These identities follow from \cite[IV (2.2.4), (2.2.5), (3.2.10)]{Ill}.
\end{rmk}

Now projecting to the first factors in the second columns of the last two diagrams, and using the definition of $\at_{\text{red}}(-)$ given in Section \ref{def-obs}, we obtain the commutativity of  the bottom and middle squares of diagram \eqref{atiyah-commut}. Since in diagram \eqref{atiyah-commut} both columns are exact triangles, the commutativity of the top square follows, and hence we have proven that the whole diagram \eqref{atiyah-commut} commutes. Now recall from Step 2 in the proof of Theorem \ref{rel}, that the maps $\alpha_i: \Fb_{f_i}\to \LL_{f_i}$ are naturally induced from the Atiyah classes $\at_{\text{red}}(\Phi_1)$ and $\at_{\text{red}}(\Phi_2)$. Hence, by the definition of the map $r$ in Step 1 of the proof, the commutativity of diagram (\ref{diagcomm}) is equivalent to the commutativity of the top square in diagram  \eqref{atiyah-commut} proven above, and hence the proof of lemma is complete. 
 
 \end{proof}
 
Now it follows that the map $\alpha_{3}$ in diagram below is induced by the commutativity of diagram (\ref{diagcomm}) and the property of exact triangles,  and therefoew this makes the whole diagram of exact triangles commutative:

 \begin{equation}\label{the-three}
\xymatrix{
\Fb_{f_2}[-1] \ar^{\alpha_2[-1]}[d]\ar[r]^r &\dL f_2^*( \Fb_{f_1}) \ar^{f_2^*( \alpha_1)}[d] \ar[r] & \cone(r)=:\Fb_f \ar@{-->}[d]^{\alpha_3} \\  \LL_{f_2}[-1] \ar[r]^{j_2} &\dL f_2^* (\LL_{f_1}) \ar[r]^{j_1}& \LL_{f},} 
\end{equation}
 where the bottom row is the exact triangle  (\ref{forget}).

\begin{prop}\label{f}
$\alpha_3:\Fb_f\to \LL_f$ is a relative perfect obstruction theory.
 \end{prop} 
 \begin{proof}
This can be deduced from the construction of $\alpha_{1}$ and $\alpha_{2}$ as perfect deformation obstruction theories. The rightmost column of Diagram \eqref{the-three} is induced by taking the cone of the map between left two columns. A long exact sequence of cohomology applied to both rows of the diagram immediately shows that $\Fb_f$ has non-vanishing cohomologies only in degrees $0,-1$. Moreover, the long exact cohomology sequence of both rows in degrees $0$ and $-1$, and application of 4-Lemma immediately shows that since $h^0(\alpha_{i}), i=1,2$ is an isomorphism, then $h^0(\alpha_{3})$ is an isomorphism. Furthermore, the surjectivity of $h^{-1}(\alpha_{i}), i=1,2$ and the 4-Lemma implies surjectivity of $h^{-1}(\alpha_{3})$. 
 \end{proof}
 \begin{prop} \label{abs-r}
The moduli space $\mathscr{M}^{(\overrightarrow{v}^{*}_{1},\overrightarrow{v}^{*}_{2}, \overrightarrow{v}^{*}_{3})}_{\text{T}}(S,q) $ is equipped with the perfect absolute obstruction theory $\Fb$ with the derived dual \begin{align}\label{full-3}&\Fbv\cong \cone \left(\left[\bigoplus_{i=1}^3 \dR\hom_\pi \left(\mathscr{E}_{i},\mathscr{E}_{i}\right)\right]_0 \xrightarrow{\tiny \left[\begin{array}{cccccc}-\Xi'_1 & \Xi_1 & 0&\\
0 & -\Xi'_2 & \Xi_2\end{array}\right]} \bigoplus_{i=1}^{2}\dR \hom_\pi \left(\mathscr{E}_{i}, \mathscr{E}_{i+1}\right)\right),\end{align}
where $[-]_0$ means the trace-free part.
\end{prop}
 \begin{proof}
 First note that by construction, for $i=1, 2$,
 $$\Fbv_{f_i}=\cone\Big(\dR\hom_\pi\left(\mathscr{E}_{i+1}, \mathscr{E}_{i+1}\right)\xrightarrow{\Xi_i} \dR \hom_\pi\left(\mathscr{E}_{i}, \mathscr{E}_{i+1}\right)\Big).$$

and consider the following two commutative diagrams
$$\mbox{\scriptsize\xymatrix@C=3.7em{\dR\hom_\pi\left(\mathscr{E}_{3}, \mathscr{E}_{3}\right )\ar[r] \ar@{=}[d]& \dR\hom_\pi\left(\mathscr{E}_{2},\mathscr{E}_{2}\right)\oplus \dR\hom_\pi\left(\mathscr{E}_{3}, \mathscr{E}_{3}\right) \ar[r] \ar[d]_-{\tiny \left[\begin{array}{cc}\Xi_1 & 0\\ -\Xi'_2 & \Xi_2\end{array} \right]} & \dR\hom_\pi\left(\mathscr{E}_{2},\mathscr{E}_{2}\right)\ar[d]_-{[\Xi_1\;-q \circ \Xi'_2]^t}\\ \dR\hom_\pi\left(\mathscr{E}_{3}, \mathscr{E}_{3}\right )\ar[r]_-{[0\; \Xi_2]^t} & \dR \hom_\pi \left(\mathscr{E}_{1}, \mathscr{E}_{2}\right) \oplus \dR \hom_\pi \left(\mathscr{E}_{2}, \mathscr{E}_{3}\right) 
 \ar[r]_-{\tiny \left[\begin{array}{cc}\id & 0\\ 0 & q\end{array} \right]} & \dR \hom_\pi \left(\mathscr{E}_{1}, \mathscr{E}_{2}\right)\oplus \cone(\Xi_2)}}$$
and$$ \xymatrix@=3.em{\cone(\Xi_1)[-1]\ar[r] \ar[d]^{r^\vee[-1]}& \dR\hom_\pi\left(\mathscr{E}_{2},\mathscr{E}_{2}\right)\ar[r]^{\Xi_1} \ar[d]_-{[\Xi_1\;-q \circ \Xi'_2]^t}& \dR \hom_\pi \left(\mathscr{E}_{1}, \mathscr{E}_{2}\right)\ar@{=}[d]\\ \cone(\Xi_2) \ar[r] & \dR \hom_\pi \left(\mathscr{E}_{1}, \mathscr{E}_{2}\right) \oplus \cone(\Xi_2)\ar[r]& \dR \hom_\pi \left(\mathscr{E}_{1}, \mathscr{E}_{2}\right)}$$
in which all four rows are natural exact triangles and $q:\dR \hom_\pi \left(\mathscr{E}_{2}, \mathscr{E}_{3}\right)\to \cone(\Xi_2)$ is the natural map . Now we can take the cone of the vertical maps in the second diagram and obtain $$\Fbv_{f}\cong \cone(r^\vee[-1])\cong \cone([\Xi_1\;-q \circ \Xi'_2]^t),$$which by the commutativity of the squares in the first diagram implies that

\begin{align}\label{rel-13}\Fbv_{f}=&\cone\bigg(\dR\hom_\pi\left(\mathscr{E}_{2}, \mathscr{E}_{2}\right)\oplus \dR\hom_\pi\left(\mathscr{E}_{3}, \mathscr{E}_{3}\right ) \\ \notag&\xrightarrow{\tiny \left(\begin{array}{cc}\Xi_1 & \Xi_2\\\Xi'_2 & 0\end{array} \right)} \dR \hom_\pi\left(\mathscr{E}_{1}, \mathscr{E}_{2}\right)\oplus \dR \hom_\pi\left(\mathscr{E}_{2}, \mathscr{E}_{3}\right) \bigg). 
 \end{align}
As in the proof of Theorem \ref{abs}, the fact that $\mathscr{M}^{\overrightarrow{v}^{*}_{1}}$ is smooth, now can be used to show that 

\begin{equation}\label{the-map2}
\Fb:=\cone(\Fb_{f}\xrightarrow{\theta} f^{*}\Omega_{\mathscr{M}^{\overrightarrow{v}^{*}_{1}}}[1])[-1],
\end{equation}
is a perfect absolute obstruction theory for $\mathscr{M}^{(\overrightarrow{v}^{*}_{1},\overrightarrow{v}^{*}_{2}, \overrightarrow{v}^{*}_{3})}_{\text{T}}(S,q) $. Moreover, using similar procedure as in Theorem \ref{abs} we can show that 
 \begin{align*}\Fbv=&\cone\bigg(\Big[\dR\hom_\pi\left(\mathscr{E}_{1}, \mathscr{E}_{1}\right )\oplus \dR\hom_\pi\left(\mathscr{E}_{2}, \mathscr{E}_{2}\right)\oplus \dR\hom_\pi\left(\mathscr{E}_{3}, \mathscr{E}_{3}\right ) \Big]_0\\ &\to \dR \hom_\pi\left(\mathscr{E}_{1}, \mathscr{E}_{2}\right)\oplus \dR \hom_\pi\left(\mathscr{E}_{2}, \mathscr{E}_{3}\right) \bigg), \end{align*}
where the arrow is as claimed in the statement of the Proposition.
 \end{proof}
\begin{cor} \label{abs-rr}
Let $\mathscr{M}^{(\overrightarrow{v}^{*}_{1},\cdots, \overrightarrow{v}^{*}_{r})}_{\text{T}}(S,q_{+_{\gg}}), r> 3$ be the projective moduli schemes of $q_{+_{\gg}}$-stable triples with fixed numerical data $(\overrightarrow{v}^{*}_{1},\cdots, \overrightarrow{v}^{*}_{r})$, satisfying the condition that for $i=1$ either $\text{rk}(E_{1})=1$ or that $\text{gcd}(r_{1}, \text{deg}(E_{1}))=1$, and for all $i>1$ we have that $r_{i}\leq r_{i+1}$. Let $S$ be a smooth projective surface such that $K_{S}\leq 0$. Then $\mathscr{M}^{(\overrightarrow{v}^{*}_{1},\cdots, \overrightarrow{v}^{*}_{r})}_{\text{T}}(S,q_{+_{\gg}})$ is equipped with a perfect absolute obstruction theory $\Fb$ with the derived dual 
\tiny{\begin{align*}&\Fbv\cong \cone \left(\left[\bigoplus_{i=1}^r \dR\hom_\pi \left(\mathscr{E}_{i},\mathscr{E}_{i}\right)\right]_0 \xrightarrow{\tiny \left[\begin{array}{cccccc}-\Xi'_1 & \Xi_1 & 0& \dots &0 &0\\
0 & -\Xi'_2 & \Xi_2 & 0 & \dots &0\\ 
\cdot & \cdot & \dots & \cdot & \cdot & \cdot \\
\cdot & \cdot & \dots & \cdot & \cdot & \cdot \\
\cdot & \cdot & \dots & \cdot & \cdot & \cdot \\
0 & 0& \dots & 0 & -\Xi'_{r-1} & \Xi_{r-1} \\
 \end{array}\right]} \bigoplus_{i=1}^{r-1}\dR \hom_\pi \left(\mathscr{E}_{i}, \mathscr{E}_{i+1}\right)\right),\end{align*}}
where $[-]_0$ means the trace-free part.
\end{cor}
\begin{proof}
This can be done by induction on $r$.
\end{proof}

\subsection{Reduced perfect obstruction theory} \label{sec:reduced}
Let us assume that $p_{g}(S)>0$, and consider a scenario in which the moduli space of $q_{+_{\gg}}$-stable flags$$E_{1}\xrightarrow{\phi_{1}} E_{2}\xrightarrow{\phi_{2}}\cdots\xrightarrow{\phi_{n-1}} E_{n}$$
satisfies the condition that there exists some $1\leq i\leq n$ such that $E_{i}$ is given by a stable torsion-free sheaf of rank 1 on $S$, i.e. a twisted  (by a Cartier divisor $D_{i}$) ideal sheaf of zero dimensional subscheme $Z_{i}$ of length $n_{i}$ in $S$. Then, by our construction of the moduli space of $q_{+_{\gg}}$-stable flags, the injectivity of morphisms $\phi_{i}, i=1,\cdots n-1$ implies that all torsion-free coherent sheaves $E_{j}, 1\leq j< i$ must be given as twisted ideal sheaves on $S$. An instance of this situation was given as the summand $\mMw_h(v)^{\C^*}_{(1,1,2)}$ in decomposition \eqref{rk=4} of the monopole branch in Section \ref{connect-VW}. Without loss of generality let us consider such scenario in the context of moduli space of stable flags of length 3. Let us assume that the chern character vectors $\overrightarrow{v}^{*}_{1},\overrightarrow{v}^{*}_{2}, \overrightarrow{v}^{*}_{3}$ are given such that $\mathscr{M}^{(\overrightarrow{v}^{*}_{1},\overrightarrow{v}^{*}_{2}, \overrightarrow{v}^{*}_{3})}_{\text{T}}(S,q_{+_{\gg}})$ parametrizes flags where the first two terms are given as twisted ideal sheaves of zero dimensional subschemes, and the third term is given by possibly a higher rank torsion-free coherent sheaf on $S$: $$I_{Z_{1}}(-D_{1})\xrightarrow{\phi_{1}} I_{Z_{2}}(-D_{2})\xrightarrow{\phi_{2}}E_{3}.$$Here $Z_{i}\in S^{[n_{i}]}, i=1,2$ are zero dimensional subschemes of $S$ of colength $n_{i}, i=1,2$, and $D_{i}, i=1,2$ are Cartier divisors in $S$. Twisting away by $D_{1}$, one can view the moduli space as the one parametrizing $$I_{Z_{1}}\xrightarrow{\phi_{1}} I_{Z_{2}}(C)\xrightarrow{\phi_{2}}E_{3},$$where both $\phi_{1}$ and $\phi_{2}$ are nonzero maps upto multiplication by scalars and in fact are injective by our choice of stability, and $C$ is an effective curve with class $\beta\in H^{2}(S, \mathbb{Z})$. Note that by construction and our choice of Chern characters above, $\mathscr{M}^{(\overrightarrow{v}^{*}_{1},\overrightarrow{v}^{*}_{2})}(S,q_{+_{\gg}})$ and $\mathscr{M}^{\overrightarrow{v}^{*}_{1}}$ are the two-step Nested Hilbert scheme $\S{n_1,n_2}_\beta$ in \cite[Definition 2.1]{GSY17a} and Hilbert scheme of points $S^{[n_{1}]}$ respectively.

We are going to argue that in our inductive construction of the perfect obstruction in Section \ref{high-length}, the relative obstruction theories in each step (c.f. Equation \eqref{forget}) will induce a surjective co-section of their relative obstruction sheaves which will force the virtual fundamental cycle associated to the final perfect obstruction theory to vanish. 
Let us make an assumption that any effective line bundle $L\in \text{Pic}(S)$ with $c_1(L)=\beta$, satisfies the condition that \begin{equation}\label{condition} |L^{-1}\otimes \omega_{S}|=\emptyset \quad \text{or equivalently} \quad H^2(L)=0. \end{equation}
By construction in Theorem \ref{rel} the relative deformation-obstruction complex for the map $f_{1}$ in$$\mathscr{M}^{(\overrightarrow{v}^{*}_{1},\overrightarrow{v}^{*}_{2})}(S,q_{+_{\gg}})\xrightarrow{f_1} \mathscr{M}^{\overrightarrow{v}^{*}_{1}}$$ is given by 
$$\Fbv_{f_{1}}\cong\cone \bigg( \dR\hom_\pi\left(\mathscr{E}_{2},\mathscr{E}_{2}\right)\xrightarrow{[(\Xi',\Xi)]} \dR \hom_\pi \left(\mathscr{E}_{1}, \mathscr{E}_{2}\right)\bigg)$$  We then get a natural morphism in the derived category $$\mu_{f_{1}}: \Fbv_{f_{1}}\to \dR\hom_\pi\left(\mathscr{E}_{2},\mathscr{E}_{2}\right)[1],$$ that induces \begin{align*}h^1(\mu_{f_{1}}): h^1(\Fbv_{f_{1}})\cong & \;\ext^2_{\pi}\big(\cE_{2}/\cE_{1},\cE_{2}\big)\to \ext^2_\pi\big (\cE_{2},\cE_{2}\big) \cong \dR^2\pi_*\O_{S\times \S{n_1,n_2}_\beta}\cong \O_{\mathscr{M}^{(\overrightarrow{v}^{*}_{1},\overrightarrow{v}^{*}_{2})}(S,q_{+_{\gg}})}^{ p_g},\end{align*} where second isomorphism is by Serre duality and stability of $\cE_{2}$. We claim now that $h^1(\mu)$ is surjective. To see this, by basechange, it suffices to prove fiberwise surjectivity of $h^1(\mu)$. Let  $t:p\hookrightarrow \mathscr{M}^{(\overrightarrow{v}^{*}_{1},\overrightarrow{v}^{*}_{2})}(S,q_{+_{\gg}})$ be the inclusion of an arbitrary closed point $p=\{I_{Z_{1}}\xrightarrow{\phi_{1}}I_{Z_{2}}(C)\}\in \mathscr{M}^{(\overrightarrow{v}^{*}_{1},\overrightarrow{v}^{*}_{2})}(S,q_{+_{\gg}})$. Then, by basechange we have the natural exact sequence\footnote{Note that $\Ext^3_S(\coker (\phi_{1}),I_{Z_2}(C))=\Ext^3_S(I_{Z_2}(C),I_{Z_2}(C))=0$.}
$$\dots \to h^1(\dL t^*\Fbv_{f_{1}})\xrightarrow{h^1(\mu_{f_{1}})_p} \Ext^2_S(I_{Z_2}(C),I_{Z_2}(C))\xrightarrow{u} \Ext^2_S(I_{Z_1},I_{Z_2}(C))\to 0.$$
The surjectivity of the morphism $u$ was established in Step 1 of the proof of Theorem \ref{rel}. We obtain 
 $$\Ext^2_S(I_{Z_2}(C),I_{Z_2}(C))\cong \Ext^2_S(I_{Z_2},I_{Z_2})\cong H^2(\O_S),$$ $$\Ext^2_S(I_{Z_1},I_{Z_2}(C))^{\vee}\cong\Hom_S(I_{Z_2},I_{Z_1}(C)\otimes \omega_{S}) \subseteq \Hom_S(I_{Z_2},\O_S(C)^D)\cong H^0(\O_S(C)^D).$$ By assumption \eqref{condition}, $H^0(\O_S(C)^D)=0$, and hence $h^1(\mu_{f_{1}})_p$ is surjective and the claim follows. Now consider the morphism $f$ is defined as the composition of $f_{1}$ and $f_{2}$ in \eqref{forget}$$\mathscr{M}^{(\overrightarrow{v}^{*}_{1},\overrightarrow{v}^{*}_{2},\overrightarrow{v}^{*}_{3})}(S,q_{+_{\gg}})\xrightarrow{f_{2}}\mathscr{M}^{(\overrightarrow{v}^{*}_{1},\overrightarrow{v}^{*}_{2})}(S,q_{+_{\gg}})\xrightarrow{f_1} \mathscr{M}^{\overrightarrow{v}^{*}_{1}}$$
 
Using the iterative cone construction in Diagram \eqref{the-three} we obtained that 
 \begin{align*}\Fbv_{f}=&\cone\bigg(\dR\hom_\pi\left(\mathscr{E}_{2}, \mathscr{E}_{2}\right)\oplus \dR\hom_\pi\left(\mathscr{E}_{3}, \mathscr{E}_{3}\right ) \\ \notag&\xrightarrow{\tiny \left(\begin{array}{cc}\Xi_1 & \Xi_2\\\Xi'_2 & 0\end{array} \right)} \dR \hom_\pi\left(\mathscr{E}_{1}, \mathscr{E}_{2}\right)\oplus \dR \hom_\pi\left(\mathscr{E}_{2}, \mathscr{E}_{3}\right) \bigg). 
 \end{align*} 
Therefore, similar to the argument above we obtain a composite morphism in the derived category $$\mu_{f}: \Fbv_{f}\to \bigg(\dR\hom_\pi\left(\mathscr{E}_{2},\mathscr{E}_{2}\right)\oplus  \dR\hom_\pi\left(\mathscr{E}_{3},\mathscr{E}_{3}\right)\bigg)[1]\to \dR\hom_\pi\left(\mathscr{E}_{2},\mathscr{E}_{2}\right)[1],$$where the second morphism is the projection onto the first factor of the middle term.  Now taking the first we obtain
\begin{align*}
&h^1(\mu_{f}): h^1(\Fbv_{f})\cong \;\text{Coker}\left(\ext^{1}(\cE_{2}, \cE_{2})\oplus \ext^{1}(\cE_{3}, \cE_{3})\to \ext^1(\cE_{1}, \cE_{2})\oplus \ext^{1}(\cE_{2}, \cE_{3})\right)\notag\\
&\to \ext^2_\pi\big (\cE_{2},\cE_{2}\big)\oplus \ext^2_\pi\big (\cE_{3},\cE_{3}\big) \to \ext^2_\pi\big (\cE_{2},\cE_{2}\big)\cong \dR^2\pi_*\O_{S\times \S{n_1,n_2}_\beta}\cong \O_{\mathscr{M}^{(\overrightarrow{v}^{*}_{1},\overrightarrow{v}^{*}_{2},\overrightarrow{v}^{*}_{3})}(S,q_{+_{\gg}})}^{ p_g},\end{align*}
Now let $t': p'\hookrightarrow \mathscr{M}^{(\overrightarrow{v}^{*}_{1},\overrightarrow{v}^{*}_{2},\overrightarrow{v}^{*}_{3})}(S,q_{+_{\gg}})$ be the inclusion of an arbitrary point $p'=I_{Z_{1}}\to I_{Z_{2}}(C)\to E_{3}$ lying in the fiber of $\mathscr{M}^{(\overrightarrow{v}^{*}_{1},\overrightarrow{v}^{*}_{2},\overrightarrow{v}^{*}_{3})}(S,q_{+_{\gg}})$ over $p\in \mathscr{M}^{(\overrightarrow{v}^{*}_{1},\overrightarrow{v}^{*}_{2}}(S,q_{+_{\gg}})$. We can then repeat the same argument and show the fiberwise surjectivity of the map $h^1(\mu_{f})$. Finally, consider the morphism $\theta$ in \eqref{the-map2}. Taking the cohomology in degree 1 we obtain the diagram
 $$ \xymatrix{\Ext^1_S(I_{Z_1},I_{Z_1})_0 \ar[r]^-{h^1(\theta^\vee)_{p'}}& h^1(\Fbv_{f}) \ar[r] \ar[d]^{h^1(\mu)_{p'}} & h^1(\dL t'^*\Fbv) \ar[r] &0 \\
& \Ext^2_S(I_{Z_2}(C),I_{Z_2}(C)) & &} $$ 
where the first row is exact by the proof Proposition \ref{abs}. But since $$h^1(\mu)_P \circ h^1(\theta^\vee)_P=0,$$ the surjection $h^1(\mu)_{p'}$ factors through $h^1(\dL t'^*\Fbv)$. Therefore since the choice of $p,p'$ are arbitrary, by basechange again there exists a surjection $$h^1(\Fbv) \to  \O_{\mathscr{M}^{(\overrightarrow{v}^{*}_{1},\overrightarrow{v}^{*}_{2},\overrightarrow{v}^{*}_{3})}(S,q_{+_{\gg}})}^{ p_g}.$$
 
\begin{prop} \label{vanish}
If the condition \eqref{condition} is satisfied and $p_g(S)>0$, then, $$[\mathscr{M}^{(\overrightarrow{v}^{*}_{1},\overrightarrow{v}^{*}_{2},\overrightarrow{v}^{*}_{3})}(S,q_{+_{\gg}})]^{\vir}=0.$$
\end{prop}
\begin{proof}
Under the assumptions of the proposition, we showed above that the obstruction sheaf admits a surjection $$h^1(\Fbv) \to  \O_{\mathscr{M}^{(\overrightarrow{v}^{*}_{1},\overrightarrow{v}^{*}_{2},\overrightarrow{v}^{*}_{3})}(S,q_{+_{\gg}})}^{ p_g}\xrightarrow{[1\;\dots \;1]}  \O_{\mathscr{M}^{(\overrightarrow{v}^{*}_{1},\overrightarrow{v}^{*}_{2},\overrightarrow{v}^{*}_{3})}(S,q_{+_{\gg}})},$$ and hence the associated virtual class vanishes by \cite[Theorem 1.1]{KL13}.
\end{proof}
\begin{rmk}
Roughly speaking, given a flag $$I_{Z_{1}}\xrightarrow{\phi_{1}} I_{Z_{2}}(C)\xrightarrow{\phi_{2}}E_{3},$$one would expect to obtain two factors of trivial bundles in the obstruction sheaf of $\Fb$ induced by $\C^*$ automorphisms of $E_{1}\cong I_{Z_{1}}$ and $E_{2}\cong I_{Z_{2}}(C)$ . We have already taken the first trivial factor out in the very construction of $\Fb$ in Proposition \ref{abs-r} by taking the trace-free part of the first term on the right hand side of \eqref{full-3}, and the procedure elaborated above proves the existence of the second trivial factor induced by the automorphisms of $E_{2}$. Hence to obtain a nontrivial virtual cycle one needs to reduce the obstruction theory by removing the latter trivial bundle of rank $p_{g}$ as will be elaborated below.
\end{rmk}

\begin{defi} \label{defired} The map $h^1(\mu_{f})$ induces the morphism in derived category $$\Fbv\to h^1(\Fbv)[-1]\to h^1(\Fbv_{f})[-1]\xrightarrow{h^1(\mu)} \O^{p_g}_{\mathscr{M}^{(\overrightarrow{v}^{*}_{1},\overrightarrow{v}^{*}_{2},\overrightarrow{v}^{*}_{3})}(S,q_{+_{\gg}})}[-1].$$ Dualizing gives a map $ \O^{p_g}_{\mathscr{M}^{(\overrightarrow{v}^{*}_{1},\overrightarrow{v}^{*}_{2},\overrightarrow{v}^{*}_{3})}(S,q_{+_{\gg}})}[1]\to \Fb$. Define $\Fb_{\red}$ to be its cone.
\end{defi}

We will now show that under a stronger condition than \eqref{condition}, $\Fb_{\red}$ gives rise to a perfect obstruction theory over $\mathscr{M}^{(\overrightarrow{v}^{*}_{1},\overrightarrow{v}^{*}_{2},\overrightarrow{v}^{*}_{3})}(S,q_{+_{\gg}})$. First note that the curve class $\beta \in H^{1,1}(S)\cap H^2(S,\ZZ)$ defines an element of $H^1(\Omega_S)$ and consider the natural pairing $T_S\otimes \Omega_S\to \O_S$. This condition is\footnote{This is condition (3) in \cite{KT14}.} 
\begin{equation} \label{scond} H^1(T_S)\xrightarrow{*\cup \beta} H^2(T_S\otimes \Omega_S)\to  H^2(\O_S)\quad \text{is surjective.}\end{equation}

To show $\Fb_{\red}$ gives rise to a perfect obstruction theory we use the procedure explained by Kool-Thomas in  \cite{KT14}. $S$ is embedded as the central fiber of an algebraic twistor family  $\cS\to B$, where $B$ is a first order Artinian neighborhood of the origin in a certain $p_g$-dimensional family of the first order deformations of $S$. More explicitly, let $$V\subset H^1(T_S)$$ be  a subspace over which $*\cup \beta$ in \eqref{scond} restricts to an isomorphism, and let $\mathfrak m$ denote the maximal ideal at the origin $0 \in H^1(T_S)$. Then, $$B:=\text{Spec} (\O_V/\mathfrak m^2),$$ and $\cS$ is the restriction of a tautological flat family of surfaces with Kodaira-Spencer class the identity in $$H^1(T_S)^*\otimes
H^1(T_S)\cong \Ext^1(\Omega_S,\O_S\otimes H^1(T_S)).$$  The Zarisiki tangent space $T_B$ is naturally identified with $V$.  By \cite[Lemma 2.1]{KT14}, $\cS$ is transversal to the Noether-Lefschetz locus of the $(1,1)$-class $\beta$, and as a result, $\beta$ does not deform outside of the central fiber of the family. Using this fact, as in \cite[Proposition 2.3]{KT14}, one can show that 
\begin{equation}\label{isomnest}\mathscr{M}^{(\overrightarrow{v}^{*}_{1},\overrightarrow{v}^{*}_{2},\overrightarrow{v}^{*}_{3})}(S,q_{+_{\gg}}) \cong \mathscr{M}^{(\overrightarrow{v}^{*}_{1},\overrightarrow{v}^{*}_{2},\overrightarrow{v}^{*}_{3})}(\cS/B,q_{+_{\gg}}),
\end{equation} where the right hand side is the relative moduli space of the family $\cS\to B$. 
\begin{rmk}
In cases where more terms in a given flag are given by twisted ideal sheaves, for instance the three-step nested Hilbert scheme $S^{[n_{1},n_{2},n_{3}]}_{\beta_{1}, \beta_{2}}$ \cite{GSY17a}, one needs to modify this construction by defining a family $\cS$ which is transverse to Noether-Lefschetz loci of more fixed (1,1)-classes, for instance in case of three step nested Hilbert scheme, $\beta_{1}$ and $\beta_{2}$.  
\end{rmk}

We denote $$\cE_{1}\to \cE_{2}\to \cE_{3}$$ to be the universal objects over $\cS\times_B  \mathscr{M}^{(\overrightarrow{v}^{*}_{1},\overrightarrow{v}^{*}_{2},\overrightarrow{v}^{*}_{3})}(\cS/B,q_{+_{\gg}})$, which by our assumption, fiberwise over $b\in B$ and $p\in \mathscr{M}^{(\overrightarrow{v}^{*}_{1},\overrightarrow{v}^{*}_{2},\overrightarrow{v}^{*}_{3})}(\cS/B,q_{+_{\gg}})$ we have that $E_{1}\cong I_{Z_{1}}$ and $E_{2}\cong I_{Z_{2}}(C)$. Now let $\pi$ be the projection to the second factor of $\cS\times_B  \mathscr{M}^{(\overrightarrow{v}^{*}_{1},\overrightarrow{v}^{*}_{2},\overrightarrow{v}^{*}_{3})}(\cS/B,q_{+_{\gg}})$. The arguments of Section \ref{2-step} can now be adapted to prove that 
\begin{align*}
\cone \left(\left[\bigoplus_{i=1}^3 \dR\hom_\pi \left(\mathscr{E}_{i},\mathscr{E}_{i}\right)\right]_0 \xrightarrow{\tiny \left[\begin{array}{cccccc}-\Xi'_1 & \Xi_1 & 0&\\
0 & -\Xi'_2 & \Xi_2\end{array}\right]} \bigoplus_{i=1}^{2}\dR \hom_\pi \left(\mathscr{E}_{i}, \mathscr{E}_{i+1}\right)\right),\end{align*} 
is the virtual tangent bundle of a perfect $B$-relative obstruction theory $ \Gb_{\rel}\to \LL_{\mathscr{M}^{(\overrightarrow{v}^{*}_{1},\overrightarrow{v}^{*}_{2},\overrightarrow{v}^{*}_{3})}(\cS/B,q_{+_{\gg}})}$, and $$\Gb:=\cone\left( \Gb_{\rel}\to   \Omega_{B}[1]\right)[-1]\to \LL_{\mathscr{M}^{(\overrightarrow{v}^{*}_{1},\overrightarrow{v}^{*}_{2},\overrightarrow{v}^{*}_{3})}(\cS/B,q_{+_{\gg}})}$$ is the associated absolute perfect obstruction theory. By the definitions of $\Fb$ and $\Gb_{\rel}$, and the isomorphism \eqref{isomnest}, we see that $\Fb\cong \Gb_{\rel}$. Now we claim that the composition $$\Gb\to \Gb_{\rel}\cong \Fb \to \Fb_{\red}$$ is an isomorphism. By the definitions of $\Gb_{\rel}$ and $\Fb_{\red}$, to prove the claim, it suffices to show that \begin{equation}\label{chain} \O^{p_g}_{\mathscr{M}^{(\overrightarrow{v}^{*}_{1},\overrightarrow{v}^{*}_{2},\overrightarrow{v}^{*}_{3})}(\cS/B,q_{+_{\gg}})} \to  \Fb[-1]\cong\Gb_{\rel}[-1]\to \Omega_B\end{equation} is an isomorphism. By the Nakayama lemma we may check this at a closed point $p=I_{Z_{1}}\xrightarrow{\phi_{1}} I_{Z_{2}}(C)\xrightarrow{\phi_{2}} E_{3}\in \mathscr{M}^{(\overrightarrow{v}^{*}_{1},\overrightarrow{v}^{*}_{2},\overrightarrow{v}^{*}_{3})}(S,q_{+_{\gg}})$. We define the reduced Atiyah class corresponding to $p$ as follows. Consider the natural homomorphisms of sheaf of graded algebras on $S$ 

 $$ \xymatrix@=1em{\O_{S} \ar[r]& \O_{S} \oplus I_{Z_{1}} \ar[r]& \O_{S}\oplus I_{Z_{2}}(C)}$$

The degree 1 part of the transitivity triangle associated to the graded cotangent complexes induced by the second row gives the first arrow in 
\begin{equation}\label{K1-grade}
\coker(\phi_{1})\to k^1\big(\mathbb{L}^{\bullet,\text{gr}}_{\O_S\oplus I_{Z_1}}\otimes (\O_S\oplus I_{Z_2(C)})\big)[1] \xrightarrow{\text{project}} \Omega_S\otimes I_{Z_{2}}(C)[1].
\end{equation}
  Define $\at_{\red}(\phi_{1})$ to be the composition of these two arrows. After dualizing \eqref{K1-grade} and using the identifications above the pullback of \eqref{chain} to $p$ becomes 
\begin{align*}T_B=V\subset H^1(T_S)\xrightarrow{ \at_{\red}(\phi_{1})}&\Ext^2(\coker (\phi),I_{Z_2}(C))\\ \xrightarrow{h^1(\mu)_p}&\Ext^2(I_{Z_2}(C),I_{Z_2}(C))\xrightarrow{\tr}H^2(\O_S).\end{align*}
Here as in \cite{KT14}, one needs to use a similar argument as \cite[Proposition 13]{MPT10} to deduce that the composition of $\Gb_{\rel}\to \LL_{\mathscr{M}^{(\overrightarrow{v}^{*}_{1},\overrightarrow{v}^{*}_{2},\overrightarrow{v}^{*}_{3})}(\cS/B,q_{+_{\gg}})}$ and the Kodaira-Spencer map $\LL_{\mathscr{M}^{(\overrightarrow{v}^{*}_{1},\overrightarrow{v}^{*}_{2},\overrightarrow{v}^{*}_{3})}(\cS/B,q_{+_{\gg}})}\to \Omega_B[1]$ for $\mathscr{M}^{(\overrightarrow{v}^{*}_{1},\overrightarrow{v}^{*}_{2},\overrightarrow{v}^{*}_{3})}(\cS/B,q_{+_{\gg}})$ coincides with the cup product of $\at_{\red}(\phi_{1})$ and the Kodaira-Spencer class for $S$. Similar to [ibid], this is achieved by relating the reduced Atiyah class of $\cS\times_B  \mathscr{M}^{(\overrightarrow{v}^{*}_{1},\overrightarrow{v}^{*}_{2},\overrightarrow{v}^{*}_{3})}(\cS/B,q_{+_{\gg}})$ from the transitivity triangle  associated to the natural maps of sheaves of graded algebras $$\O_B\to \O_{\cS\times_B  \mathscr{M}^{(\overrightarrow{v}^{*}_{1},\overrightarrow{v}^{*}_{2},\overrightarrow{v}^{*}_{3})}(\cS/B,q_{+_{\gg}})}\oplus \i{n_1}\to \O_{\cS\times_B  \mathscr{M}^{(\overrightarrow{v}^{*}_{1},\overrightarrow{v}^{*}_{2},\overrightarrow{v}^{*}_{3})}(\cS/B,q_{+_{\gg}})}\oplus \i{n_2}_\beta$$ to the reduced Atiyah classes of each factors.
\begin{lem} $h^1(\mu)_P\circ\at_{\red}(\phi_{1})=\at(I_{Z_2}(C)),$ where $$\at(I_{Z_2}(C))\in \Ext^1(I_{Z_2}(C), \Omega_S\otimes I_{Z_2}(C))$$ is the usual Atiyah class.  \end{lem} 
\begin{proof}
Consider the commutative diagram of sheaves of graded algebras with all unlabelled arrows are the obvious natural maps:
$$
 \xymatrix@=2em{\C \ar[r] & \O_S\oplus I_{Z_{1}} \ar[r]^-{(\id, \phi_{1})} & \O_{S} \oplus I_{Z_{2}}(C)\\
 \C \ar[r] \ar@{=}[u] & \O_{S}  \ar[r] \ar[u]& \O_{S}\oplus I_{Z_{2}}(C).\ar@{=}[u]
}$$
Taking the degree 1 part of the the transitivity triangles of the rows followed by a projection as  in the definition of $\at_{\red}(\phi_{1})$ above we get the commutative diagram
$$
 \xymatrix@=3em{\coker(\phi) \ar[r]^-{\at_{\red}(\phi_{1})} & \Omega_S \otimes I_{Z_{2}}(C)[1]\\
 I_2(C) \ar[r]^-{\at(I_2(C))} \ar[u]^-{h^1(\mu)_P} &\Omega_S \otimes I_{Z_{2}}(C)[1] \ar@{=}[u]
}$$ proving the lemma.
\end{proof}

But by \cite[Prop 4.2]{BFl03}, $$\tr \circ \at(I_{Z_2}(C))=- *\cup \beta,$$ which by condition \eqref{scond} is an isomorphism when restricted to $V\subset H^1(T_S)$, and hence the claim is proven. We have shown

\begin{prop} \label{reduced}
If the condition \eqref{scond} is satisfied, then, $\Fb_{\red}$ is a perfect obstruction theory on $\S{n}_\beta$, and hence defines a reduced virtual fundamental class
$$[\mathscr{M}^{(\overrightarrow{v}^{*}_{1},\overrightarrow{v}^{*}_{2},\overrightarrow{v}^{*}_{3})}(S,q_{+_{\gg}})]^{\vir}_{\red}\in A_{d'}(\mathscr{M}^{(\overrightarrow{v}^{*}_{1},\overrightarrow{v}^{*}_{2},\overrightarrow{v}^{*}_{3})}(S,q_{+_{\gg}})), \quad d'=rk(\Fb)+p_g(S).$$
\end{prop}\qed

\noindent{\tt{artan@cmsa.fas.harvard.edu,Center for Mathematical Sciences and\\ Applications, Harvard University, Department of Mathematics, 20 Garden Street, Room 207, Cambridge, MA, 02139}}\\\\
\noindent{\tt{Centre for Quantum Geometry of Moduli Spaces, Aarhus University, Department of Mathematics
Ny Munkegade 118, building 1530, 319, 8000 Aarhus C, Denmark}}\\\\
\noindent{\tt{National Research University Higher School of Economics, Russian Federation, Laboratory of Mirror Symmetry, NRU HSE, 6 Usacheva str.,Moscow, Russia, 119048}}\\\\
\noindent{\tt{yau@math.harvard.edu\\
Department of Mathematics, Harvard University, Cambridge, MA 02138, USA}}\\\\
\noindent{\tt{Center for Mathematical Sciences and Applications, Harvard University, Department of Mathematics, 20 Garden Street, Room 207, Cambridge, MA, 02139}}

\end{document}